\documentclass[11pt,a4paper]{article}

\usepackage{amsmath}
\usepackage{amssymb}
\usepackage{makeidx}
\usepackage{enumitem}

\usepackage{xcolor}
\usepackage[normalem]{ulem}

\usepackage{graphicx}

\usepackage{a4wide}

\usepackage{soul}

\newtheorem{theorem}{\bf Theorem}[section]
\newtheorem{corollary}[theorem]{\bf Corollary}
\newtheorem{definition}[theorem]{\bf Definition}
\newtheorem{lemma}[theorem]{\bf Lemma}
\newtheorem{proposition}[theorem]{\bf Proposition}

\newtheorem{example}[theorem]{Example}

\newtheorem{remark}[theorem]{Remark}

\newenvironment{proof}[1][Proof]{\noindent\textbf{#1.} }{\ \rule{0.5em}{0.5em} \medskip }

\usepackage{titling}

\usepackage{authblk}

\usepackage[breaklinks,plainpages=false,hyperfootnotes=false]{hyperref}
\hypersetup{
  colorlinks   = true, 
  urlcolor     = blue, 
  linkcolor    = red, 
  citecolor   = blue 
}

\author[1]{M. Guerra \footnote{E-mail address: \href{mailto:mguerra@iseg.ulisboa.pt}{mguerra@iseg.ulisboa.pt}}}
\author[1]{A. B. de Moura \footnote{E-mail address:
    \href{mailto:amoura@iseg.ulisboa.pt}{amoura@iseg.ulisboa.pt}}}
\affil[1]{ISEG - School of Economics and Management,
  Universidade de Lisboa; REM – Research in Economics and Mathematics, CEMAPRE}
\title{Reinsurance of multiple risks with generic dependence structures}
\date{\vspace{-6ex}}

\begin{document}

\maketitle
\allowdisplaybreaks

\begin{abstract}
We consider the optimal reinsurance problem from the point of view of a direct insurer owning several dependent risks, assuming a maximal
expected utility criterion and independent negotiation of reinsurance for each risk.
Without any particular hypothesis on the dependency structure, we show that optimal treaties
exist in a class of independent randomized contracts. We derive optimality conditions and show
that under mild assumptions the optimal contracts are of classical (non-randomized) type.
A specific form of the optimality conditions applies in that case.
We present a numerical scheme to solve the optimality conditions.
\end{abstract}

\vspace{3mm}

\vspace{3mm}


\section{Introduction}\label{S introduction}

Reinsurance constitutes a risk mitigation strategy and
an essential tool in risk management. 
By transferring part of the risk to the reinsurer, the cedent company
seeks a trade-off between profit, which is reduced by the reinsurance
premium, and safety, which is increased by the reduction in exposure to
the underlying risk.
A large amount of works can be found in literature concerning optimal
reinsurance strategies, as this problem has been for long considered in the
actuarial community. The first works date as far back as the 60s, with
the seminal paper of Borch \cite{B60}.
It has been the subject of active research not only in the
field of actuarial science, where the risk transfer contract between
two insurance companies is usually analysed within a one-period
setting, see for instance 
\cite{
G04, 
CW12, 
CLT17, AC19, HW19},
but
also in the financial mathematical context, within  a 
dynamical framework usually in conjunction with investment
strategies, see for instance
\cite{
CT16, 
GVY18}. We refer to \cite{ABT17} for a comprehensive overview of the
literature on optimal reinsurance.

In this work we consider the optimal reinsurance problem within a
one-period setting and from the cedent's perspective, the
interest of the reinsurer being enclosed in the calculation
principles considered for the reinsurance premium. 
The goal is
the maximization of an expected concave utility of the insurer's
wealth.

In most works regarding optimal reinsurance, independence is assumed. 
Indeed, for many years dependence has not been considered in the study  of optimal risk transfer, possibly due to its complexity.
Nowadays, it is widely accepted that dependencies play a significant role in risk management and  several works can be found accounting for this fact.
One of the first works including the effects of dependence when investigating optimal forms or risk transfer is \cite{C05}.
There, dependencies of two classes of insurance businesses, through the number of claims, are included by means of a bivariate Poisson.
Other authors have considered the optimal reinsurance problem under dependence between claim numbers, such as \cite{ZHD15}, in a one-period setting, and \cite{BLX16}, in a dynamic setting. 
In \cite{CSY14}, the authors do not assume any particular dependence structure, as they argue it is often difficult to determine it.
They propose instead to use a minimax optimal reinsurance decision formulation, in which the worst-case scenario is first identified.
In \cite{CW12},
positive dependencies in the individual risk are considered by means of the stochastic ordering.

Due to the analytical complexity of problems including dependencies, many works propose numerical frameworks \cite{BCZ13,ABCHK17,AGHK17,ZJLC16}.
It is also worth mentioning the empirical approach proposed in \cite{TW2014} and \cite{SWZ17}, where reinsurance models are formulated based on data observation, without explicitly assuming the distribution of the underlying risks, which allows for the use of programming procedures to obtain an optimal solution. 

Stop-loss  was found to be optimal in several works which consider the expected value premium principle and various optimality criteria, both in the independent \cite{B60,G04,CSYY14} and in the dependent cases \cite{CW12,ZHD15}.

Most authors include various constraints on the type of reinsurance contracts being considered.
For example, in \cite{B60,G04,ZHD15} the optimal solution is sought among combinations of quota-share and stop-loss treaties.
A very common constraint, included in most references above, is that the retained risk after reinsurance is an increasing function of the underlying risk.
This is imposed to avoid solutions allowing for moral hazard, but other constraints with different motivations may also be considered.

In this work we will impose no further constraints besides that the ceded risk should be positive and should not exceed the underlying risk. 
We hope that by characterizing the solutions of the problem free of constraints we contribute to a better understanding of the necessity, role and impact of any constraints that might be introduced.
The analysis of the impact of constraints in the optimal solution will be the subject of future work.

Most literature on this topic considers forms or reinsurance of a deterministic nature, in the sense that for a given risk $X$, solutions are seek in a set of functions of $X$ defining for each value of loss, how much of that loss is ceded to the reinsurer.
This is the traditional and intuitive way of formulating the problem.
However, it has been shown that randomized treaties may outperform deterministic ones, at least with respect to the criterion of ruin probability.
More precisely, in
\cite{G04} the authors observe that randomized reinsurance contracts may lead to lower ruin probability than deterministic ones, when an upper constraint on the  reinsurance premium is imposed.
The randomization improvement decreases as the probability mass of every atom decreases.
The randomized decision can also be eliminated by  adjusting the upper limit on the price of reinsurance.
In \cite{AC19} the authors explore the potential of randomizing reinsurance treaties to obtain optimal reinsurance strategies minimizing VaR under the
expected value premium principle.
They provide a possible interpretation for the randomization of the reinsurance treaty as the default risk of the reinsurer, since the indemnity may not be paid with a certain probability.
In that paper, it is also argued that randomized treaties pose less moral hazard issues since, due to randomness, it is unclear \textit{a priori} who will have to pay the claim.
The authors study the possibility of implementing additional randomness in the settlement of risk transfer and show that randomizing the classical stop-loss can be beneficial for the insurer. 
In \cite{GC12}, randomization of the reinsurance treaties is used as a mathematical tool to find treaties minimizing several quantile risk measures when premia are calculated by a coherent risk measure. 
In the present work, randomized contracts will also serve as a mathematical tool to prove existence of the optimal contract when the underlying risks are dependent through an arbitrary joint distribution.
To the best of our knowledge, this is the first time randomized treaties are analysed under dependencies. 

In this work we consider the optimal reinsurance problem of $n$ dependent risks.
By risk we mean the aggregate claims of a line of business, a portfolio of policies or a policy. 
The dependence structure is arbitrary, defined through a generic joint distribution function.
We assume that reinsurance is negotiated separately (independently) for each risk, and each premium is calculated by a (possibly different) function of moments of the ceded risk. 
The cedent's criterion is the maximization of the expected value of a concave utility function of the overall retained risk, net of reinsurance premia.
We introduce the class of conditionally independent randomized strategies and show that it contains an optimal strategy.
Optimality conditions are obtained, and it is shown that under mild conditions, the optimal strategy is deterministic.

This paper is organized as follows.
In Section \ref{S optimization problem}, we formulate the optimization problem.
In Section \ref{S existence} the class of randomized reinsurance strategies is introduced together with the proper probability spaces, and the existence result is proved. 
In Section  \ref{S optimality conds}, we provide necessary optimality conditions.
In Section \ref{S deterministic} it is shown that under very general conditions the optimal treaty is deterministic.
We conclude the paper in Section \ref{S algorithm} with the presentation of a general numerical scheme to solve the optimality conditions. Some numerical examples are given at the end of the section.

\section{The optimization problem}\label{S optimization problem}

We consider a portfolio of
$n\geq 2$
risks. 
Let
$X_{i}$,
$i=1,2,...,n$
denote the aggregate value of claims placed under the $i$-th risk on a given period of time (say, one year).
$X=\left(X_1,X_2,...,X_n\right) $
is a non-negative random vector with joint probability law
$\mu_X$.

The direct insurer acquires a reinsurance policy for each risk separately.
Each of these policies is a measurable function
$Z_i$
such that
\begin{equation}
\Pr \left\{ 0\leq Z_i\left( X_i \right) \leq X_i \right\} =1.  \label{Z1}
\end{equation}
For each
$i=1,2,...,n$,
let
$\mathcal{Z}_{i}$
denote the set of measurable functions satisfying (\ref{Z1}), and let
$\mathcal{Z}=\prod\limits_{i=1}^{n} \mathcal{Z}_{i}$, be
the cartesian product of all
$\mathcal{Z}_{i}$.

For each risk, the corresponding policy is priced by a functional
$P_{i}: \mathcal{Z}_{i}\mapsto \left[ 0,+\infty \right] $,
depending only on the probability law of
$Z_i(X_i)$.
In this paper, we assume that these premia calculation principles are of type
\begin{equation}
\label{Eq premium moments}
P_i(Z_i) = \Psi_i \left( \mathbb E Z_i ,\mathbb E Z_i^2 , \ldots , \mathbb E Z_i^{k_i} \right) ,
\end{equation}
where
$\Psi_i:[0,+\infty[^{k_i} \mapsto [0,+\infty[$,
$i=1,2,\ldots, n$
are continuous functions.

The insurer's net profit after reinsurance is
\begin{equation}
\label{Eq net profit}
L_{Z}=c-\sum_{i=1}^{n}\left( P_{i}\left( Z_{i}\right) +X_{i}-Z_{i}\left( X_{i}\right) \right) ,
\end{equation}
where
$c$
is the portfolio's aggregate premium income, net of non-claim refunding expenses.
Thus,
$L_Z$
is a random variable taking values in the interval
$]-\infty,c]$.

We assume that the insurer aims to choose a reinsurance strategy
$Z=(Z_1,Z_2, \ldots , Z_n) \in \mathcal Z$
maximizing the expected utility of net profit, i.e., maximizing the functional
\begin{equation}
\label{Eq expected utility}
\rho(Z) = \mathbb EU(L_Z) 
\qquad Z \in \mathcal Z ,
\end{equation}
where
$U:]-\infty,c] \mapsto \mathbb R$
is a concave nondecreasing function.
Nondecreasing 	monotonicity of $U$ reflects the fact that higher net profit (smaller net loss) is preferred to lower net profit (bigger net loss), while concavity of $U$ introduces risk aversion:
larger losses are valued more, compared to smaller losses.
This effect increases with the ratio $-\frac{U^{\prime \prime}}{U^\prime}$, which is usually called the {\it coefficient of (absolute) risk aversion}, and in the case of the expeonential utility function $U(x)=-e^{-Rx}$, it is constant and equal to the parameter $R$.
By assuming a risk aversion utility function we are considering that the insurance company searches not only to maximize profit, but also to avoid excessive risk exposure.
This behaviour explains the demand for reinsurance as reinsurance mitigates risk but comes at the price of the reinsurance premium, decreasing the profits.

\section{Existence of optimal reinsurance strategy}\label{S existence}

Under the formulation above, existence of an optimal reinsurance strategy in the class
$\mathcal Z$
can not in general be guaranteed.
However, existence in the larger class of {\it random treaties}, defined below, can easily be proved.
Our approach to the problem outlined in the previous section will be to obtain optimality conditions for random treaties and then discuss conditions under which such conditions can only be satisfied by classical treaties of class
$\mathcal Z$.

\subsection{Random treaties} \label{SS random treaties}

First, let us introduce some notation.
For any array
$x=(x_1,x_2, \ldots ,x_n)$,
we will use the usual notation
$x_i$
to denote the
$i^{\mathrm{th}}$
element of
$x$,
and the notation
$x_{[i]}$
to denote the array of
$n-1$
elements obtained from
$x$
by deleting the element
$x_i$,
i.e.
\[
x_{[i]} = (x_1, \ldots, x_{i-1},x_{i+1}, \ldots, x_n ) .
\]

With the notation above, random treaties are defined as follows:

\begin{definition}
\label{D independent random treaties}
A
$\mathbb R^n$-valued random variable
$Z=(Z_1,Z_2, \ldots, Z_n)$
is said to be a vector of (conditionally independent) random treaties or a randomized strategy if the following conditions hold for
$i=1,2, \ldots ,n$:
\begin{enumerate}
\item
$\Pr \left\{ 0 \leq Z_i \leq X_i \right\}= 1$;
\item
The random variable
$Z_i$
is conditionally independent of the random vector
$\left( X_{[i]},Z_{[i]} \right)$,
given
$X_i$.
\end{enumerate}
\end{definition}
The second condition in Definition \ref{D independent random treaties} enforces the assumption that reinsurance is acquired separately for different risks:
given the value of $X_i$, the value of claims and refunds on other risks have no bearing on the value of the refund $Z_i$. 
However, $Z_i$
and
$Z_j$ are not, in general, independent random variables, due to dependency between $X_i$
and
$X_j$.

Reinsurance strategies of class $\mathcal Z$ discussed in Section \ref{S optimization problem} are called {\it deterministic strategies} to distinguish them from randomized strategies defined above.
Notice that for any deterministic strategy
$Z \in \mathcal Z$,
the random vector
$\left( Z_1(X_1), Z_2(X_2), \ldots , Z_n(X_n) \right)$
satisfies Definition \ref{D independent random treaties}.
Thus, the class of deterministic strategies is a subset of the class of randomized strategies.

If expressions \eqref{Eq premium moments}, \eqref{Eq net profit} and
\eqref{Eq expected utility} are well defined for every deterministic strategy, then they are also well defined for every vector of random treaties. 
Thus, we may consider the problem of maximizing \eqref{Eq expected utility} over all vectors of conditionally independent random treaties.

\subsection{Spaces of probability laws}

Since \eqref{Eq expected utility} depends only on the probability law of the random vector
$(X,Z)$, 
we may discuss optimization in terms of probability laws instead of the random vectors inducing such laws.
To this purpose, we will frame our argument in canonical spaces, i.e., we will consider the underlying measurable space to be $\mathbb R^{2n}$ provided with its Borel $\sigma$-algebra, $\mathcal B_{\mathbb R^{2n}}$.
Random variables are Borel-measurable functions $\varphi: \mathbb R^{2n} \mapsto \mathbb R$, and the space of Borel probability measures $\nu: \mathcal B_{\mathbb R^{2n}} \mapsto [0,1]$ is denoted by $\mathcal P$.
Expectations with respect to a particular probability law $\nu \in \mathcal P$ are
\[
\mathbb E^\nu \varphi = \int_{\mathbb R^{2n}} \varphi(u) \, \nu(du) ,
\]
provided the integral exists.
The space
$\mathcal P$
is provided with its weak topology, that is, a sequence
$\{\nu_j \in \mathcal P \}_{j \in \mathbb N}$
is said to converge to
$\nu \in \mathcal P$
if and only if
\[
\lim_{j \mapsto \infty} \mathbb E^{\nu_j} \varphi 
=
\mathbb E^\nu \varphi 
\]
for every continuous bounded 
$\varphi$.

Let $\mathcal P_X$ be the set of all $\nu \in \mathcal P$ such that the marginal probability 
\[
\nu_X(A) = \nu \left( A \times \mathbb R^n \right)
\qquad
A \in \mathcal B_{\mathbb R^n}
\]
is equal to $\mu_X$, the joint probability law of the claim amounts, introduced in Section \ref{S optimization problem}.
We write $\nu \in \mathcal H_X$ if and only if $\nu$ is the probability law of a random vector $(X,Z)$ where $X$ is the vector of claim amounts and $Z$ is some randomized strategy in the sense of Definition \ref{D independent random treaties}.
It is clear that $\mathcal H_X \subset \mathcal P_X \subset \mathcal P$.

Coordinates in $\mathbb R^{2n}$ are indicated as $(x,z) = (x_1, \ldots, x_n, z_1, \ldots , z_n)$.
For a given $\nu \in \mathcal P$, $\nu_X$ and $\nu_Z$ indicate the marginal probability laws:
\[
\nu_X(A) = \nu \left(A \times \mathbb R^n \right),
\quad
\nu_Z(A) = \nu \left( \mathbb R^n \times A \right),
\qquad A \in \mathcal B_{\mathbb R^n}.
\]
Similar notation is used for other marginal probabilities, like $\nu_{X_i}$, $\nu_{(X_i,Z_i)}$, $\nu_{(X_{[i]},Z_{[i]})}$, etc..

Given $\nu \in \mathcal P$ having a strictly positive density function $f$, the conditional probability law of (say) $(X_{[i]},Z_{[i]})$ given $(X_i,Z_i)$ is the family of probability measures
\begin{align*}
\nu_{\left( X_{[i]}, Z_{[i]} \right) | (X_i,Z_i)} (A) = 
\frac
{\int_A f(x,z) dx_{[i]} dz_{[i]}}
{\int_{\mathbb R^{2n-2}} f(x,z) dx_{[i]} dz_{[i]}} 
\qquad
\forall A \in \mathcal B_{\mathbb R^{2n-2}} .
\end{align*}
Notice that, for fixed $A \in \mathcal B_{\mathbb R^{2n-2}}$, $\nu_{\left( X_{[i]}, Z_{[i]} \right) | (X_i,Z_i)} (A) $ is a measurable function of $(x_i,z_i)$ defined up to null subsets of $\mathbb R^2$.

In the following, we will need to consider probability laws which do not have a density function and/or do not have full support.
Thus, we will consider regular conditional probability laws.
For a given $\nu \in \mathcal P$, a mapping $\nu_{\left( X_{[i]}, Z_{[i]} \right) | (X_i,Z_i)}=Q: \mathbb R^2 \times \mathcal B_{\mathbb R^{2n-2}} \mapsto [0,1]$ is a regular conditional probability law of $\left( X_{[i]}, Z_{[i]} \right) $ given $(X_i,Z_i)$ if it satisfies the following conditions:
\begin{enumerate}[label=(\roman*)]
\item
For every
$B \in \mathcal B_{\mathbb R^{2n-2}}$
(fixed), the map
$(x_i,z_i) \mapsto Q((x_i,z_i),B)$
is measurable with respect to the Borel
$\sigma$-algebra
of
$\mathbb R^2$;

\item\label{Cd measure}
There is a set
$A \in \mathcal B_{\mathbb R^2}$
such that
$\nu_{(X_i,Z_i)} \left( A \right) = 1$
and for each
$(x_i,z_i) \in A$
(fixed) the map
$B \mapsto Q((x_i,z_i),B)$
is a probability measure in
$\mathcal B_{\mathbb R^{2n-2}}$;

\item
For every
$A \in \mathcal B_{\mathbb R^2}$,
$B \in \mathcal B_{\mathbb R^{2n-2}}$,
\begin{align*}
&
\int_A Q((x_i,z_i),B) \, \nu_{(X_i,Z_i)}(d(x_i,z_i)) =
\nu \left\{ (x,z) \in \mathbb R^{2n}:(x_i,z_i) \in A, \ (x_{[i]},z_{[i]}) \in B \right\} .
\end{align*}
\end{enumerate}
Since
$\left( \mathbb R^{2n}, \mathcal B_{\mathbb R^{2n}} \right)$
is a standard measurable space, every $\nu \in \mathcal P$ admits a regular conditional probability law
$\nu_{(X_{[i]},Z_{[i]})|(X_i,Z_i)}$ (see e.g. \cite{Cinlar11}, Theorem IV.2.7).

Conditional expectation of a random variable $\varphi: \mathbb R^{2n} \mapsto \mathbb R$ is defined as
\[
\mathbb E^\nu_{(X_{[i]},Z_{[i]})|(X_i,X_i)} \varphi =
\int_{\mathbb R^{2n-2}} \varphi(x,z) \, Q(x_i,z_i,d(x_{[i]},z_{[i]})),
\]
provided the integral on the right-hand side is well defined for
$\nu_{(X_i,Z_i)}$-almost every
$(x_i,z_i) \in \mathbb R^2$.
In that case, it is a measurable function of $(x_i,z_i)$.

If conditioned and conditioning coordinates do not span the whole space $\mathbb R^{2n}$, then the conditional probability law is defined as above with respect to the subspace spanned by the coordinates concerned. 
For example $\nu_{(X_{[i]},Z_{[i]}) |X_i}:\mathbb R \times \mathcal B_{\mathbb R^{2n-2}} \mapsto [0,1]$ is defined with respect to the marginal probability law $\nu_{(X,Z_{[i]})}: \mathcal B_{\mathbb R^{2n-1}} \mapsto [0,1]$ instead of the joint probability law $\nu$.
Conditional expectations of random variables $\varphi=\varphi(x,z_{[i]})$ are defined as above, but the conditional expectation operator $\varphi \mapsto \mathbb E_{\left.\left( X_{[i]},Z_{[i]} \right) \right|X_i}^\nu \varphi$ acts on random variables depending on the full range of coordinates $(x,z)$ by
\[
\mathbb E^\nu_{(X_{[i]},Z_{[i]})|X_i} \varphi =
\int_{\mathbb R^{2n-2}} \varphi(x,z) \, \tilde Q (x_i,d(x_{[i]},z_{[i]})).
\]
This is a measurable function of $(x_i,z_i) \in \mathbb R^2$, while 
\[
\mathbb E^\nu_{X_{[i]}|X_i} \varphi =\int_{\mathbb R^{n-1}} \varphi (x,z) \hat Q (x_i,dx_{[i]})
\]
is a function of $(x_i,z) \in \mathbb R^{1+n}$.
$\tilde Q: \mathbb R \times \mathcal B_{\mathbb R^{2n-2}} \mapsto [0,1]$, and
$\hat Q: \mathbb R \times \mathcal B_{\mathbb R^{n-1}} \mapsto [0,1]$ are regular conditional probability laws in $\mathbb R^{2n-1}$ and $\mathbb R^n$, respectively.

Recall that two $\sigma$-algebras,
$\mathcal A_1$ and $\mathcal A_2$,
are conditionally independent given a $\sigma$-algebra $\mathcal B$ if and only if
$\Pr \left( \left. A \right| \sigma ( \mathcal A_1 \cup \mathcal B)\right) = \Pr \left( \left. A \right| \mathcal B \right)$
for every
$A \in \mathcal A_2$
(see e.g. \cite{Cinlar11}, Proposition IV.3.2).
Taking Definition \ref{D independent random treaties} into account, 
$\mathcal H_X$
can be characterized as the set of all
$\nu \in \mathcal P$
satisfying the following conditions:
\begin{enumerate}[label=(\roman*)]
\item
$\nu_X= \mu_X$;
\item
$\nu$
is concentrated on the set
\[
\left\{ (x,z) \in \mathbb R^{2n} : 0 \leq z_i \leq x_i \ \ i=1,2,\ldots , n \right\} ;
\]
\item
$\nu_{(X_{[i]},Z_{[i]}) | (X_i,Z_i)} =
\nu_{(X_{[i]},Z_{[i]}) | X_i}, \qquad i=1,2,\ldots, n$,

i.e., for every $A \in \mathcal B_{\mathbb R^{2n-2}}$, the function $(x_i,z_i) \mapsto \nu_{(X_{[i]},Z_{[i]}) | (X_i,Z_i)}(x_i,z_i,A)$ depends only on $x_i$.
\end{enumerate}

The premium calculation principles \eqref{Eq premium moments} can be extended to 
the set of probability measures $\nu \in \mathcal P$
such that $\mathbb E^\nu Z_i^{k_i} < +\infty$ through the obvious expression
\begin{equation}
\label{Eq premium moments relaxed}
P_i(\nu) = \Psi_i \left( \mathbb E^\nu Z_i, \mathbb E^\nu Z_i^2, \ldots , \mathbb E^\nu Z_i^{k_i} \right) .
\end{equation}
Similarly, the net profit random variable is the function $L_\nu: \mathbb R^{2n} \mapsto \mathbb R$ defined as
\begin{equation}
\label{Eq net profit relaxed}
L_\nu (x,z) = c - \sum_{i=1}^n \left( P_i(\nu) + x_i - z_i \right) ,
\end{equation}
and the expected utility functional \eqref{Eq expected utility} can be extended to 
$\nu \in \mathcal P$
by
\begin{equation}
\label{Eq expected utility relaxed}
\rho ( \nu) = \mathbb E^\nu U \left( L_\nu \right).
\end{equation}

The following proposition provides an important reason to consider the space
$\mathcal H_X$
instead of
$\mathcal Z$.

\begin{proposition}
\label{P compact H}
$\mathcal H_X$
is a relatively compact subset of
$\mathcal P$.
\\	
$\mathcal H_X$ is compact if the following assumption holds:
\begin{description}
\item[{\bf (AC)}]
$\mu_X$ is absolutely continuous with respect to the product of marginal distributions $\mu_{X_1} \times \mu_{X_2} \times \ldots \times \mu_{X_n}$.
\end{description}
\end{proposition}

To prove the second part of Proposition \ref{P compact H}, we will use the following addenda to the Portmanteau theorem on weak convergence:

\begin{lemma}
\label{L portmanteau}
Let $S_i$, $i=1,2, \ldots, n$, be separable metric spaces, provided with the respective Borel $\sigma$-algebras, $\mathcal B_{S_i}$.
\\
A sequence of probability measures $\left\{ \nu_j : \bigotimes\limits_{i=1}^n \mathcal B_{S_i} \mapsto [0,1] \right\}_{j \in \mathbb N}$ converges weakly to $\nu : \bigotimes\limits_{i=1}^n \mathcal B_{S_i} \mapsto [0,1]$ if and only if
\[
\lim_{j \to \infty} \nu_j\left(\prod\limits_{i=1}^n A_i \right) = 
\nu \left(\prod\limits_{i=1}^n A_i \right)
\]
for every $A_i \in \mathcal B_{S_i}$ 
such that 
$\nu \left( \left( \prod\limits_{\ell < i}S_\ell \right) \times \left(\partial A_i \right) \times \left( \prod\limits_{\ell > i}S_\ell \right)  \right) =  0$, $i=1,2, \ldots, n$ ( $\bigotimes$ denoting the product of $\sigma$-algebras).
\end{lemma}
\begin{proof}
The fact that the condition is necessary follows trivially from the Portmanteau theorem (see e.g. \cite[Chapter 1]{Billingsley99}), since 
$\nu \left( \left( \prod\limits_{\ell < i}S_\ell \right) \times \left(\partial A_i \right) \times \left( \prod\limits_{\ell > i}S_\ell \right)  \right) =  0$, $i=1,2, \ldots, n$,
obviously implies $\nu \left( \partial \left( \prod\limits_{i=1}^n A_i \right) \right) =  0$ 
\\
To prove that the condition is sufficient, assume that it holds and fix an open set $U \subset \prod\limits_{i=1}^n S_i$.
Then, there is a countable family $\left\{ \prod\limits_{i=1}^n A_{i,k} \right\}_{k \in \mathbb N}$, such that every $A_{i,k}\subset S_i$ is open, 
$\nu \left( \left( \prod\limits_{\ell < i}S_\ell \right) \times \left(\partial A_{i,k} \right) \times \left( \prod\limits_{\ell > i}S_\ell \right)  \right) =  0$, and $U = \bigcup\limits_{k=1}^\infty \prod\limits_{i=1}^n A_{i,k}$.
Thus, there is a sequence $\left\{ \prod\limits_{i=1}^n B_{i,k} \in
\bigotimes\limits_{i=1}^n \mathcal B_{S_i} \right\}_{k \in \mathbb N}$ such that 
$\nu \left( \left( \prod\limits_{\ell < i}S_\ell \right) \times \left(\partial B_{i,k} \right) \times \left( \prod\limits_{\ell > i}S_\ell \right)  \right) =  0$,
$U = \bigcup\limits_{k=1}^\infty \prod\limits_{i=1}^n B_{i,k}$, and $\left(\prod\limits_{i=1}^n B_{i,k} \right) \cap \left(\prod\limits_{i=1}^n B_{i,\ell} \right) = \emptyset$ whenever $k \neq \ell$.
Therefore,
\begin{align*}
\liminf_{j \to \infty} \nu_j(U) \geq &
\liminf_{j \to \infty} \nu_j \left( \bigcup_{\ell=1}^k \prod\limits_{i=1}^n B_{i,\ell}  \right) =
\sum_{\ell=1}^k \lim_{j \to \infty} \nu_j \left( \prod\limits_{i=1}^n B_{i,\ell}  \right) =
\sum_{\ell=1}^k \nu \left( \prod\limits_{i=1}^n B_{i,\ell}  \right) 
\\ = &
\nu \left( \bigcup_{\ell=1}^k \prod\limits_{i=1}^n B_{i,\ell}  \right) .
\end{align*}
Since $\nu (U)= \lim\limits_{k \to \infty} \nu \left( \bigcup\limits_{\ell=1}^k \prod\limits_{i=1}^n B_{i,\ell}  \right) $, this shows that $\liminf\limits_{j \to \infty} \nu_j(U) \geq \nu(U)$ holds for every open $U \subset \prod \limits_{i=1}^n S_i$. Hence, the Portmanteau theorem guarantees that $\{ \nu_j\}$ converges weakly to $\nu$.
\end{proof}

\begin{proof}[Proof of Proposition \ref{P compact H}]
It follows immediately from the definition of
$\mathcal H_X$
that 
\[
\nu \left([0,M]^{2n} \right) = \mu_X \left( [0,M]^n \right) ,
\]
for any
$\nu \in \mathcal H_X$
and
$M\in [0,+\infty[$.
Since
$\lim\limits_{M \rightarrow +\infty} \mu_X \left([0,M]^n \right) = 1$,
this shows that
$\mathcal H_X $
is uniformly tight.
Hence, Prohorov's theorem (see e.g. \cite{Billingsley99}) states that
$ \mathcal H_X$
is a relatively compact subset of $\mathcal P$.

Now, suppose that $\mu_X$ is absolutely continuous with respect to $\mu_{X_1} \times \mu_{X_2} \times \ldots \times \mu_{X_n}$, i.e., there is a measurable function $\alpha: \mathbb R^n \mapsto [0,+\infty[$ such that
\[
\mu_X (dx) = 
\alpha (x)
\mu_{X_1}(dx_1) \mu_{X_2}(dx_2) \ldots \mu_{X_n}(dx_n).
\]
Fix $\nu \in \partial \mathcal H_X$, and pick a sequence $\{ \nu_j \in \mathcal H_X \}_{j \in \mathbb N}$, converging weakly to $\nu$.
Notice that $\nu \in \mathcal H_X$ if and only if the following conditions hold:
\begin{align}
& \label{Eq compact H 1}
\nu \{ (x,z): 0 \leq z_i \leq x_i, \ i=1,2,\ldots , n \}=1,
\\ & \label{Eq compact H 2}
\nu (d(x,z))=
\nu_{Z_1|X_1}(x_1,dz_1)
\nu_{Z_2|X_2}(x_2,dz_2)
\ldots
\nu_{Z_n|X_n}(x_n,dz_n) \mu_X(dx),
\end{align}
where $\nu_{Z_i|X_i}: \mathbb R \times \mathcal B_{\mathbb R} \mapsto [0,1]$, $i=1,2,\ldots, n$, are regular conditional probabilities.

Since for every $j \in \mathbb N$, $\nu_j$ satisfies condition \eqref{Eq compact H 1} and the set $\{(x,z): 0 \leq z_i \leq x_i, \ i\leq n\}$ is closed, we only need to prove that 
condition \eqref{Eq compact H 2} holds.
Notice that given $\nu$, the conditional probability measures $\nu_{Z_i|X_i}$ are well defined and the right-hand side of \eqref{Eq compact H 2} is a probability measure, which we denote by $\tilde \nu$.
We are going to prove that $\{\nu_j \}$ converges weakly to $\tilde \nu$.

For every $\varepsilon>0$ there is a continuous function with compact support $\alpha_\varepsilon: \mathbb R^n \mapsto [0,+\infty[$ such that
\[
\int_{\mathbb R^n} \left|\alpha (x)- \alpha_\varepsilon (x) \right| \, \mu_{X_1}(dx_1) \mu_{X_2}(dx_2) \ldots \mu_{X_n}(dx_n) < \varepsilon .
\]
For any $\eta >0$, the support of $\alpha_\varepsilon$ admits a finite cover by rectangles $\prod\limits_{i=1}^n ]a_{i,k}^\eta, b_{i,k}^\eta [$, $k=1,2,\ldots ,K$ such that
\[
0 < b_{i,k}^\eta - a_{i,k}^\eta < \eta, \quad
\mu_{X_i}\{a_{i,k}^\eta,b_{i,k}^\eta\} =0,
\qquad 
i = 1,2, \ldots , n, \ k=1,2, \ldots, K.
\]
Thus, the support of $\alpha_\varepsilon$ admits a finite cover by measurable rectangles $\prod\limits_{i=1}^n C_{i,m}^\eta$, $m=1,2,\ldots , M$, such that
\begin{align*}
&
\mathrm{diam}(C_{i,m}^\eta) < \eta,
\quad 
\mu_{X_i}\left(\partial C_{i,m}^\eta \right) =0
\qquad
\text{for } i \leq n, \ m\leq M,
\\ &
\left( \prod_{i=1}^n C_{i,m}^\eta\right)
\cap
\left( \prod_{i=1}^n C_{i,\ell}^\eta\right) = \emptyset \qquad \text{when } m \neq \ell.
\end{align*}

Pick sets $A_1,A_2, \ldots, A_n,B_1,B_2, \ldots, B_n \in \mathcal B_{\mathbb R}$ such that 
\begin{align*}
&
\nu (\partial (\mathbb R^{i-1} \times A_i \times \mathbb R^{2n-i})) = 
\nu (\partial (\mathbb R^{n+i-1} \times B_i \times \mathbb R^{n-i})) = 0,
\qquad
i=1,2, \ldots, n.
\end{align*}
Then, choose 
$\prod\limits_{i=1}^n C_{i,m}^\eta$, $m=1,2,\ldots , M$, as above, with $\eta$ sufficiently small such that
$| \alpha_\varepsilon (x) - \alpha_\varepsilon (y) | < \varepsilon$ whenever $|x-y|< n \eta$, and picking points $x^{(m)} \in \prod\limits_{i=1}^n C_{i,m}^\eta$, $m=1,2,\ldots , M$.
Then:
\begin{align*}
&
\left| \tilde \nu \left(\left( \prod_{i=1}^n A_i \right) \times \left( \prod_{i=1}^n B_i \right) \right) - \nu_j \left(\left( \prod_{i=1}^n A_i \right) \times \left( \prod_{i=1}^n B_i \right) \right) \right| =
\\ = &
\Bigg|
\int_{\prod\limits_{i=1}^n A_i}
\alpha (x) \nu_{Z_1|X_1} (x_1,B_1)
\ldots
\nu_{Z_n|X_n} (x_n,B_n) \,
\mu_{X_1}(dx_1) \ldots \mu_{X_n}(dx_n) -
\\ & -
\int_{\prod\limits_{i=1}^n A_i}
\alpha (x)
\left( \nu_j \right)_{Z_1|X_1} (x_1,B_1)
\ldots
\left( \nu_j \right)_{Z_n|X_n} (x_n,B_n) \,
\mu_{X_1}(dx_1) \ldots \mu_{X_n}(dx_n) 
\Bigg| 
\\ \leq &
\Bigg|
\int_{\prod\limits_{i=1}^n A_i}
\alpha_\varepsilon (x) \nu_{Z_1|X_1} (x_1,B_1)
\ldots
\nu_{Z_n|X_n} (x_n,B_n) \,
\mu_{X_1}(dx_1) \ldots \mu_{X_n}(dx_n) -
\\ & -
\int_{\prod\limits_{i=1}^n A_i}
\alpha_\varepsilon (x)
\left( \nu_j \right)_{Z_1|X_1} (x_1,B_1)
\ldots
\left( \nu_j \right)_{Z_n|X_n} (x_n,B_n) \,
\mu_{X_1}(dx_1) \ldots \mu_{X_n}(dx_n) 
\Bigg| 
+ 2 \varepsilon 
\\ = &
\Bigg|
\sum_{m=1}^M \int_{\prod\limits_{i=1}^n A_i \cap C_{i,m}^\eta}
\alpha_\varepsilon (x) \nu_{Z_1|X_1} (x_1,B_1)
\ldots
\nu_{Z_n|X_n} (x_n,B_n) \,
\mu_{X_1}(dx_1) \ldots \mu_{X_n}(dx_n) -
\\ & -
\sum_{m=1}^M \int_{\prod\limits_{i=1}^n A_i \cap C_{i,m}^\eta}
\alpha_\varepsilon (x)
\left( \nu_j \right)_{Z_1|X_1} (x_1,B_1)
\ldots
\left( \nu_j \right)_{Z_n|X_n} (x_n,B_n) \,
\mu_{X_1}(dx_1) \ldots \mu_{X_n}(dx_n) 
\Bigg| 
+ 2 \varepsilon 
\\ \leq &
\Bigg|
\sum_{m=1}^M \alpha_\varepsilon (x^{(m)}) \int_{\prod\limits_{i=1}^n A_i \cap C_{i,m}^\eta}
\nu_{Z_1|X_1} (x_1,B_1)
\ldots
\nu_{Z_n|X_n} (x_n,B_n) \,
\mu_{X_1}(dx_1) \ldots \mu_{X_n}(dx_n) -
\\ & -
\sum_{m=1}^M \alpha_\varepsilon (x^{(m)}) \int_{\prod\limits_{i=1}^n A_i \cap C_{i,m}^\eta}
\left( \nu_j \right)_{Z_1|X_1} (x_1,B_1)
\ldots
\left( \nu_j \right)_{Z_n|X_n} (x_n,B_n) \,
\mu_{X_1}(dx_1) \ldots \mu_{X_n}(dx_n) 
\Bigg| 
+ 4 \varepsilon
\\ \leq &
\sum_{m=1}^M \alpha_\varepsilon (x^{(m)}) \Bigg|
\prod_{i=1}^n
\int_{A_i \cap C_{i,m}^\eta}
\nu_{Z_i|X_i} (x_i,B_i) \,
\mu_{X_i}(dx_i) -
\\ & -
\prod_{i=1}^n \int_{A_i \cap C_{i,m}^\eta}
\left( \nu_j \right)_{Z_i|X_i} (x_i,B_i) \,
\mu_{X_i}(dx_i) 
\Bigg| 
+ 4 \varepsilon 
\\ = &
\sum_{m=1}^M \alpha_\varepsilon (x^{(m)}) \Bigg|
\prod_{i=1}^n \nu_{(X_i,Z_i)} \left(A_i \cap C_{i,m}^\eta\right) -
\prod_{i=1}^n \left( \nu_j \right)_{(X_i, Z_i)} \left(A_i \cap C_{i,m}^\eta\right)
\Bigg| 
+ 4 \varepsilon 
.
\end{align*}
Since $\{\left( \nu_j \right)_{(X_i, Z_i)}\}_{j \in \mathbb N}$ converges weakly to $\nu_{(X_i, Z_i)}$ (for $i=1,2, \ldots, n$), the Portmanteau theorem states that 
$\lim\limits_{j \to \infty} \left( \nu_j \right)_{(X_i, Z_i)} \left(A_i \cap C_{i,m}^\eta\right) = \nu_{(X_i, Z_i)} \left(A_i \cap C_{i,m}^\eta\right)$, and therefore
\[
\limsup_{j \to \infty} \left| \tilde \nu \left(\left( \prod_{i=1}^n A_i \right) \times \left( \prod_{i=1}^n B_i \right) \right) - \nu_j \left(\left( \prod_{i=1}^n A_i \right) \times \left( \prod_{i=1}^n B_i \right) \right) \right| \leq 4 \varepsilon.
\]
Since $\varepsilon$ is arbitrary, this shows that
\[
\lim_{j \to \infty} \nu_j \left(\left( \prod_{i=1}^n A_i \right) \times \left( \prod_{i=1}^n B_i \right) \right) =
\tilde \nu \left(\left( \prod_{i=1}^n A_i \right) \times \left( \prod_{i=1}^n B_i \right) \right).
\]
Therefore, Lemma \ref{L portmanteau} guarantees that $\{ \nu_j\}$ converges weakly to $\tilde \nu$.
\end{proof}

The following example shows that assumption {\bf (AC)} cannot be dropped from Proposition \ref{P compact H}, even in cases where the marginal distributions are absolutely continuous.

\begin{example}
\label{Ex noncompact H}
Consider the case $n=2$, with $\mu_X(d(x_1,x_2)) = \delta_{x_1}(dx_2) f(x_1)\, dx_1$, where $\delta_a$ denotes the Dirac measure concentrated at the point $a$, and $f:\mathbb R \mapsto [0,+\infty[$ is a continuous function such that $\int_0^{+\infty} f(x_1)\, dx_1=1$.
Thus, the marginal distributions are absolutely continuous ($\mu_{X_i}(dx_i) = f(x_i)\, dx_i$ for $i=1,2$) but $\mu_X$ is not absolutely continuous with respect to $\mu_{X_1} \times \mu_{X_2}$.
\\
Pick continuous functions $Z_{1,1}, Z_{1,2}, Z_{2,1}, Z_{2,2} : [0,+\infty[ \mapsto [0,+\infty[$ such that
\begin{align*}
&
0 \leq Z_{i,k}(x_i) \leq x_i \qquad \forall x_i \geq 0, \ \ i,k \in \{1,2 \},
\\ &
Z_{1,1}(x_1) \neq Z_{1,2}(x_1), \quad \text{and} \quad
Z_{2,1}(x_2) \neq Z_{2,2}(x_2) \qquad \forall x_1, x_2 >0 ,
\end{align*}
and consider the deterministic treaties
\begin{align*}
&
Z_i^{(j)}(x_i) = \left\{
\begin{array}{ll}
Z_{i,1}(x_i) & \text{for } x_i \in \left[ \frac{k}{2^j}, \frac{k+1}{2^j} \right[ \ \text{with $k$ even},
\\ 
Z_{i,2}(x_i) & \text{for } x_i \in \left[ \frac{k}{2^j}, \frac{k+1}{2^j} \right[ \ \text{with $k$ odd},
\qquad
i =1,2, \ \ j \in \mathbb N.
\end{array}
\right.
\end{align*}
The strategies $\left(Z_1^{(j)}, Z_2^{(j)}\right)$ are represented by the measures $\nu_j \in \mathcal H_X$, defined as
\begin{align*}
&
\nu_j(d(x,z)) = 
\\ = &
\left(
\sum\limits_{ k =0}^{\infty} \left(
\chi_{ \left[ \frac{2k}{2^j},\frac{2k+1}{2^j} \right[ } (x_1) \delta_{Z_{1,1}(x_1)} +
\chi_{ \left[ \frac{2k+1}{2^j},\frac{2k+2}{2^j} \right[ } (x_1) \delta_{Z_{1,2}(x_1)}
\right)
\right) (dz_1)
\\&
\left(
\sum\limits_{ \ell =0}^{\infty} \left(
\chi_{ \left[ \frac{2\ell}{2^j},\frac{2\ell+1}{2^j} \right[ } (x_1) \delta_{Z_{2,1}(x_2)} +
\chi_{ \left[ \frac{2\ell+1}{2^j},\frac{2\ell+2}{2^j} \right[ } (x_1) \delta_{Z_{2,2}(x_2)}
\right) 
\right) (dz_2) \, \delta_{x_1}(dx_2) \, f(x_1) \, dx_1,
\end{align*}
where $\chi_A$ denotes the indicator function of set $A$.
For any continuous bounded function $\varphi: \mathbb R^4 \mapsto \mathbb R$, we have
\begin{align*}
\mathbb E^{\nu_j} \varphi =&
\sum_{\substack{k \ \text{even} \\ \ell  \ \text{even}}}
\int_{ \frac{k}{2^j}}^{ \frac{k+1}{2^j} } \int_{ \frac{\ell}{2^j}}^{ \frac{\ell+1}{2^j} }
\varphi \left(x_1,x_2, Z_{1,1}(x_1),Z_{2,1}(x_2)\right)
\delta_{x_1}(dx_2) \, f(x_1) \, dx_1 +
\\ & +
\sum_{\substack{k \ \text{even} \\ \ell  \ \text{odd}}}
\int_{ \frac{k}{2^j}}^{ \frac{k+1}{2^j} } \int_{ \frac{\ell}{2^j}}^{ \frac{\ell+1}{2^j} }
\varphi \left(x_1,x_2, Z_{1,1}(x_1),Z_{2,2}(x_2)\right)
\delta_{x_1}(dx_2) \, f(x_1) \, dx_1 +
\\ & +
\sum_{\substack{k \ \text{odd} \\ \ell  \ \text{even}}}
\int_{ \frac{k}{2^j}}^{ \frac{k+1}{2^j} } \int_{ \frac{\ell}{2^j}}^{ \frac{\ell+1}{2^j} }
\varphi \left(x_1,x_2, Z_{1,2}(x_1),Z_{2,1}(x_2)\right)
\delta_{x_1}(dx_2) \, f(x_1) \, dx_1 +
\\ & +
\sum_{\substack{k \ \text{odd} \\ \ell  \ \text{odd}}}
\int_{ \frac{k}{2^j}}^{ \frac{k+1}{2^j} } \int_{ \frac{\ell}{2^j}}^{ \frac{\ell+1}{2^j} }
\varphi \left(x_1,x_2, Z_{1,2}(x_1),Z_{2,2}(x_2)\right)
\delta_{x_1}(dx_2) \, f(x_1) \, dx_1 
\\ = &
\sum_{ k \ \text{even}}
\int_{ \frac{k}{2^j}}^{ \frac{k+1}{2^j} } \varphi \left(x_1,x_1, Z_{1,1}(x_1),Z_{2,1}(x_1)\right)
\, f(x_1) \, dx_1 +
\\ & +
\sum_{ k \ \text{odd}}
\int_{ \frac{k}{2^j}}^{ \frac{k+1}{2^j} } \varphi \left(x_1,x_1, Z_{1,2}(x_1),Z_{2,2}(x_1)\right)
\, f(x_1) \, dx_1 .
\end{align*}
Since the integrands are continuous, this shows that
\begin{align*}
\lim_{j \to \infty} \mathbb E^{\nu_j } \varphi = &
\frac 1 2
\int_0^{ +\infty } \varphi \left(x_1,x_1, Z_{1,1}(x_1),Z_{2,1}(x_1)\right)
\, f(x_1) \, dx_1 +
\\ & +
\frac 1 2
\int_0^{ +\infty } \varphi \left(x_1,x_1, Z_{1,2}(x_1),Z_{2,2}(x_1)\right)
\, f(x_1) \, dx_1 
.
\end{align*}
That is, $\{\nu_j \}$ converges weakly to
\begin{align*}
&
\nu(d(x,z)) = 
\frac 1 2
\left(
\delta_{Z_{1,1}(x_1)}(dz_1) \delta_{Z_{2,1}(x_2)}(dz_2) +
\delta_{Z_{1,2}(x_1)}(dz_1) \delta_{Z_{2,2}(x_2)}(dz_2)
\right) 
\, \delta_{x_1}(dx_2) \, f(x_1) \, dx_1.
\end{align*}
Since
\begin{align*}
&
\nu_{Z_1|X_1}(x_1,dz_1) = 
\frac 1 2
\left( \delta_{Z_{1,1}(x_1)} +
\delta_{Z_{1,2}(x_1)}
\right) (dz_1) ,
\\ &
\nu_{Z_1|(X_1,X_2,Z_2)}(x_1,x_2,z_2,dz_1) = 
\left(
\chi_{ \{Z_{2,1}(x_1)\} }(z_2) \delta_{Z_{1,1}(x_1)} +
\chi_{ \{Z_{2,2}(x_1)\} }(z_2)
\delta_{Z_{1,2}(x_1)}
\right) (dz_1),
\end{align*}
we see that $\nu \notin \mathcal H_X$.

A similar computation shows that in the case $\mu_X(dx) = f(x_1, x_2) \, dx_1 \, dx_2$, with $f$ continuous and everything else identical to the above, $\{ \nu_j \}$ converges weakly to
\[
\nu(d(x,z)) = 
\frac 1 4 
\left( \delta_{Z_{1,1}(x_1)} + \delta_{Z_{1,2}(x_1)} \right) (dz_1) \,
\left( \delta_{Z_{2,1}(x_2)} + \delta_{Z_{2,2}(x_2)} \right) (dz_2) \,
f(x_1,x_2) \, dx_1 \, dx_2 .
\]
In this case,
\begin{align*}
&
\nu_{Z_1|(X_1,X_2,Z_2)}(x_1,x_2,z_2,dz_1) =
\nu_{Z_1|X_1}(x_1,dz_1) =
\frac 1 2 
\left( \delta_{Z_{1,1}(x_1)} + \delta_{Z_{1,2}(x_1)} \right) (dz_1),
\\ &
\nu_{Z_2|(X_1,X_2,Z_1)}(x_1,x_2,z_1,dz_2) =
\nu_{Z_2|X_2}(x_2,dz_2) =
\frac 1 2 
\left( \delta_{Z_{2,1}(x_2)} + \delta_{Z_{2,2}(x_2)} \right) (dz_2),
\end{align*}
and therefore $\nu \in \mathcal H_X$.
\end{example}

\subsection{Existence of random maximizers}

Now, we prove that the functional \eqref{Eq expected utility relaxed} admits a maximizer in the space 
$\overline{\mathcal H}_X$ (the closure of $\mathcal H_X$).
Notice that under assumption {\bf (AC)}, $\mathcal H_X = \overline{\mathcal H}_X$, and therefore the optimal reinsurance problem admits a solution in the class of (conditionally independent) randomized strategies.
If assumption {\bf (AC)} fails, then  optimal random treaties still exist, but they may fail to be conditionally independent, though they can be approximated by sequences of conditionally independent randomized strategies.

We will need
the following result concerning moments of ceded risks.

\begin{proposition}
\label{P continuity of momenta}
If $\mathbb E X_i^k < \infty$, then the functional $\nu \mapsto \mathbb E^\nu Z_i^k $ is continuous in 
$\overline{\mathcal H}_X$.
\end{proposition}

\begin{proof}
Consider a sequence $\left\{ \nu_j \in 
\overline{\mathcal H}_X
\right\}_{j \in \mathbb N}$, converging weakly to some $\nu \in
\overline{\mathcal H}_X
$.

By weak convergence, for every $M<+\infty $ we have
\begin{align*}
\mathbb E^\nu \left( Z_i \wedge M \right)^k =
\lim_{j \rightarrow \infty} \mathbb E^{\nu_j} \left( Z_i \wedge M \right)^k \leq
\liminf_{j \rightarrow \infty} \mathbb E^{\nu_j} Z_i^k 
.
\end{align*}
Since $\mathbb E^\nu Z_i^k = \lim\limits_{M \rightarrow \infty} \mathbb E^\nu \left( Z_i \wedge M \right)^k$, this proves that
\[
\mathbb E^\nu Z_i^k \leq \liminf_{ j \rightarrow \infty} \mathbb E^{\nu_j}Z_i^k .
\]

Since every $\nu \in \overline{\mathcal H}_X$ satisfies \eqref{Eq compact H 1},  
$\mathbb E \left[X_i^k \right] < + \infty$ implies $\lim\limits_{M \rightarrow +\infty} \mathbb E^\nu \left[ X_i ^k \chi_{\{X_i>M\}} \right] = 0$.
Therefore, 
\begin{align*}
&
\limsup_{j \rightarrow \infty} \mathbb E^{\nu_j} Z_i^k \leq 
\limsup_{j \rightarrow \infty} \mathbb E^{\nu_j} \left[ \left( Z_i\wedge M\right)^k + X_i^k \chi_{\{X_i>M\} } \right] 
\\ = &
\lim_{j \rightarrow \infty} \mathbb E^{\nu_j} \left( Z_i \wedge M\right)^k +
\mathbb E^\nu \left[ X_i^k \chi_{\{X_i>M\} } \right] =
\mathbb E^\nu \left( Z_i \wedge M\right)^k +
\mathbb E^\nu \left[ X_i^k \chi_{\{X_i>M\} } \right].
\end{align*}
Hence,
\[
\limsup\limits_{j \rightarrow \infty} \mathbb E^{\nu_j} Z_i^k \leq \mathbb E^\nu Z_i^k .
\]
\end{proof}

We will also need the following lemma.

\begin{lemma}
\label{L location of moments}
For any $k \in \mathbb N$, let
\[
D_k = 
\left\{
x \in ]0,+\infty[^k : x_1 < x_2^{\frac 12} < x_3^{\frac 1 3} < \cdots < x_k^{\frac 1k} 
\right\} .
\]
For any non-negative random variable $Y$ such that $\mathbb E Y^k < + \infty$:
\begin{enumerate}
\item 
$\left( \mathbb EY, \mathbb EY^2, \ldots , \mathbb EY^k \right) \in D_k \cup \left\{
x \in [0,+\infty[^k : x_1 = x_2^{\frac 12} = x_3^{\frac 1 3} = \cdots = x_k^{\frac 1k} 
\right\} $.
\item 
$\mathbb EY=\left( \mathbb EY^2\right)^{\frac 12}=\left( \mathbb EY^3\right)^{\frac 13}= \ldots = \left( \mathbb EY^k \right)^{\frac 1k}$ if and only if $Y$ is degenerate, i.e., if and only if it takes a constant value almost surely.
\end{enumerate}
\end{lemma}

\begin{proof}
Pick $i,j$ such that $1 \leq i < j \leq k$.
Using H\"older's inequality, we obtain
\begin{align*}
\mathbb EY^i \leq
\left( \mathbb E 1^{\frac j{j-i}} \right)^{\frac{j-i}j} \left( \mathbb E Y^j \right)^{\frac ij} =
\left( \mathbb E Y^j \right)^{\frac ij},
\end{align*}
with the equality holding if and only if $\{ 1, Y^i \}$ are linearly dependent.
\end{proof}

The main result in this section is the following.

\begin{theorem}
\label{T existence random}
\label{T existence of random maximizers}
If $\mathbb E X_i^{k_i}< \infty$ for $i=1,2, \ldots, n$ 
and the functions $\Psi_i$, $i=1,2,\ldots, n$ are continuous in $D_{k_i} \cup \left\{
x \in [0,+\infty[^{k_i} : x_1 = x_2^{\frac 12} = x_3^{\frac 1 3} = \cdots = x_{k_i}^{\frac 1{k_i}} 
\right\} $, then the expected utility functional \eqref{Eq expected utility relaxed} is upper semicontinuous in
$\overline{\mathcal H}_X$
and therefore admits a maximizer.
\end{theorem}

\begin{proof}
Unter the assumptions of the theorem, the premia $\nu \mapsto P_i(\nu) = \Psi_i \left(\mathbb E^\nu Z,\mathbb E^\nu Z^2, \ldots, \mathbb E^\nu Z^{k_i}\right)$ are continuous in $\overline{ \mathcal H}_X$.

Fix a sequence $\left\{ \nu_j \in
\overline{\mathcal H}_X
\right\}_{j \in \mathbb R}$ converging weakly to some $\nu \in \overline{\mathcal H}_X$.
By weak convergence, we have
\begin{align*}
\limsup_{j \to \infty} \mathbb E^{\nu_j} U \left( c - \sum_{i=1}^n \left(P_i (\nu_j) +X_i-Z_i\right)\right) \leq &
\lim_{j \to \infty} \mathbb E^{\nu_j} \left(M \vee U \left( c - \sum_{i=1}^n \left(P_i (\nu_j) +X_i-Z_i\right)\right)\right)
\\ = &
\mathbb E^\nu \left(M \vee U \left( c - \sum_{i=1}^n \left(P_i (\nu) +X_i-Z_i\right)\right)\right),
\end{align*}
for every $M \in ]-\infty, 0[$. 
By Lebesgue's monotome convergence theorem,
\begin{align*}
\mathbb E^\nu U \left( c - \sum_{i=1}^n \left(P_i (\nu_j) +X_i-Z_i\right)\right) =
\lim_{M \to -\infty} \mathbb E^\nu \left(M \vee U \left( c - \sum_{i=1}^n \left(P_i (\nu) +X_i-Z_i\right)\right)\right),
\end{align*}
and therefore the functional \eqref{Eq expected utility relaxed} is upper semicontinuous in $\overline{\mathcal H}_X$.

Existence of a maximizer follows by Weierstrass' theorem.
\end{proof}

Contrasting with existence, uniqueness of optimal randomized strategies seems difficult to prove in a general setting.
This may be surprising since, for any functional $\Phi: \mathcal Z \mapsto \mathbb R$ of type $\Phi (Z) = \mathbb E \left( \phi (X,Z(X) )\right)$, the extended functional $\Phi: \overline{\mathcal H}_X \mapsto \mathbb R$ is convex, provided the function $\phi: \mathbb R^{2n} \mapsto \mathbb R$ satisfies integrability conditions guaranteeing that the relevant expectations exist.
However, there are convex functionals $\Phi: \mathcal Z \mapsto \mathbb R$ such that the extension into $\mathcal H_X$ fails to be convex.
To see this, consider the variance functional
\[
\Phi (Z_i) = \mathbb E(Z_i(X)^2) - \left(\mathbb E Z_i(X)\right)^2 ,
\]
which is strictly convex in $\mathcal Z_i$.
A simple computation shows that its extension into $\mathcal H_X $ satisfies
\begin{align*}
\Phi ((1-\alpha) \eta + \alpha \nu) = &
\mathbb E^{(1-\alpha) \eta + \alpha \nu} Z_i^2 - \left(\mathbb E^{(1-\alpha) \eta + \alpha \nu} Z_i\right)^2 
\\ = &
(1-\alpha) \Phi (\eta) + \alpha \Phi (\nu) + \alpha(1-\alpha)\left( \mathbb E^\eta Z_i -\mathbb E^\nu Z_i \right)^2,
\end{align*}
for every $\nu, \eta \in \mathcal H_X$ and $\alpha \in ]0,1[$.
Thus, the variance is concave but not convex in $\mathcal H_X$.

Further, concavity of the utility function means that convexity of the premia functionals $P_i: \mathcal Z \mapsto \mathbb R$ guarantees concavity of the optimality criterion $\rho: \mathcal Z \mapsto [-\infty,c]$, but convexity of the premia functionals $P_i: \mathcal H_X \mapsto \mathbb R$ is not sufficient to guarantee concavity of $\rho : \mathcal H_X \mapsto [-\infty, c]$.
For example, if the utility function is exponential and the premia are computed by expect value principles, then the map $\nu \mapsto \mathbb E^\nu U(L_\nu)$ fails to be convex in $\mathcal H_X$.
We leave the computations proving this nonintuitive fact in Appendix \ref{Appendix concavity}.

Finally, let us mention that in Section \ref{S deterministic}, we provide conditions guaranteing that the optimal strategy must be deterministic.
In that case, it is straightforward to argument with concavity of the optimality criterion over the set of deterministic strategies to prove uniqueness of the optimal strategy.

\section{Optimality conditions}\label{S optimality conds}

For any $i \in \{1,2,\ldots, n\}$, $(\hat x_i, \hat z_i) \in \mathbb R^2$, and $\varepsilon \geq 0$, let $B_{i,\varepsilon} =B_{i,\varepsilon}(\hat x_i, \hat z_i)$ denote the cylinder
\begin{align*}
B_{i,\varepsilon} = 
\left\{ (x,z) \in \mathbb R^{2n} : (x_i- \hat x_i)^2 + (z_i-\hat z_i)^2 \leq \varepsilon^2 \right\}
\end{align*}

Fix 
$\nu \in \overline{\mathcal H}_X$.
We consider perturbations of $\nu$, that is, probability laws $\tilde \nu = \nu^{i,\hat x_i, \hat z_i,\alpha, \varepsilon}$ defined as
\begin{align*}
\tilde \nu(A) = &
\nu \left( A \setminus B_{i,\varepsilon} \right) +
\nu \left\{ (x,z) \in B_{i, \varepsilon} : (x,z+\alpha e_i) \in A \right\}
,
\end{align*}
where $e_i$ represents the canonical unit vector in direction $i$.

Before we proceed, we need to introduce a few lemmas.

\begin{lemma}
\label{L perturbation}
If $0 \leq \hat z_i < \hat x_i$ (resp., $0 < \hat z_i \leq \hat x_i$) and $0 \leq \alpha < \hat x_i - \hat z_i$ (resp., $\hat z_i-\hat x_i < \alpha \leq 0$) and $\nu \in \mathcal H_X$, then $\tilde \nu \in \mathcal H_X$ for every sufficiently small $\varepsilon \geq 0$.
Under the same conditions, if $\nu \in \overline{\mathcal H}_X$, then $\tilde \nu \in \overline{\mathcal H}_X$ for every sufficiently small $\varepsilon \geq 0$.
\end{lemma}

\begin{proof}
Consider the case $0 \leq \hat z_i < \hat x_i$ and $0 \leq \alpha < \hat x_i - \hat z_i$.
\\
For $\varepsilon < \hat x_i - \hat z_i - \alpha$, $\tilde \nu $ satisfies
\[
\tilde \nu \{(x,z): 0 \leq z_i\leq x_i, i=1,2,\ldots , n\}=1 .
\]
Further, if $Z_i$ and $(X_{[i]},Z_{[i]})$ are conditionally independent given $X_i$ with respect to $\nu$, then they are also conditionally independent with respect to $\tilde \nu$.
Thus, we only need to prove the second statement in the Lemma.
\\
Fix $\nu \in \overline{\mathcal H}_X$, and a sequence $\{\nu_j \in \mathcal H_X \}_{j \in \mathbb N}$, converging weakly to $\nu$.
Notice that
\[
\tilde \nu(A) = 
\int_{ \mathbb R^{2n}}
\chi_A(x,z) \left( 1 - \chi_{B_{i,\varepsilon}}(x,z) \right) +
\chi_A(x,z+ \alpha e_i ) \chi_{B_{i,\varepsilon}}(x,z)  \, \nu (d(x,z)) .
\]
For every $\eta>0$, there is a continuous function $\gamma_\eta: \mathbb R^{2n} \mapsto [0,1]$ such that
\[
\chi_{ B_{i,\varepsilon}} (x,z) \leq
\gamma_\eta (x,z) \leq
\chi_{ B_{i,\varepsilon+ \eta}} (x,z) \qquad
\forall (x,z) \in \mathbb R^{2n}.
\]
We introduce the measures $\hat \nu_{j,\eta}$, defined as
\[
\hat \nu_{j,\eta }(A) = 
\int_{ \mathbb R^{2n}}
\chi_A(x,z) \left( 1 - \gamma_\eta(x,z) \right) +
\chi_A(x,z+ \alpha e_i ) \gamma_\eta(x,z)  \, \nu_j (d(x,z)) .
\]
Notice that $\hat \nu_{j,\eta} \in \mathcal H_X$ whenever $\varepsilon+ \eta < \hat x_i- \hat z_i - \alpha $.
\\
Since $\gamma_\eta$ is continuous, for every continuous bounded functions $\varphi: \mathbb R^{2n} \mapsto \mathbb R$, we have
\begin{align*}
\lim_{j \to \infty} \mathbb E^{\hat \nu_{j,\eta }} \varphi = 
\int_{ \mathbb R^{2n}} \varphi(x,z)(1 - \gamma_\eta (x,z)) +
\varphi(x,z+\alpha e_i)\gamma_\eta (x,z) \, \nu (d(x,z)).
\end{align*}
That is, for every sufficiently small $\eta>0$, the measure
\[
\hat \nu_\eta (A) = 
\int_{ \mathbb R^{2n}}
\chi_A(x,z) \left( 1 - \gamma_\eta(x,z) \right) +
\chi_A(x,z+ \alpha e_i ) \gamma_\eta(x,z)  \, \nu (d(x,z)) 
\]
is an element of $\overline{\mathcal H}_X$.
Since $\lim\limits_{\eta \to 0^+} \gamma_\eta = \chi_{ B_{i,\varepsilon}}$ pointwise, it follows from Lebesgue's dominated convergence theorem that $
\lim\limits_{\eta \to 0^+} \mathbb E^{\hat \nu_\eta} \varphi = \mathbb E^\nu \varphi
$ for every bounded continuous $\varphi: \mathbb R^{2n} \mapsto \mathbb R$, and therefore $\nu \in \overline{\mathcal H}_X$.
\\
The case $0 < \hat z_i \leq \hat x_i$ differs only in the bounds for $\varepsilon$ and $\eta$.
\end{proof}

\begin{lemma}
\label{L perturbation of moments}
If $\mathbb E X_i^k <+\infty$ then, for any $\nu \in \overline{\mathcal H}_X$:
\begin{align}
& \label{Eq perturbation of moments}
\left( \mathbb E^{\tilde \nu} - \mathbb E^\nu \right)Z_i^k =
\alpha \nu \left( B_{i,\varepsilon} \right) \left( k \hat z_i^{k-1} + O(|\alpha|) + O(\varepsilon)\right) ,
\end{align}
when $\alpha \to 0$
and
$\varepsilon \to 0^+$.
\end{lemma}

\begin{proof}
It follows from the definition of $\tilde \nu$ that
\begin{align*}
	\left( \mathbb E^{\tilde \nu} - \mathbb E^\nu \right)Z_i^k = &
	\int_{B_{i,\varepsilon}} \left( (z_i+\alpha)^k - z_i^k \right) d \nu .
\end{align*}
Thus, for $k=1$, \eqref{Eq perturbation of moments} holds with $\left(\mathbb E^{\tilde \nu} - \mathbb E^\nu \right)Z_i = \alpha \nu \left(B_{i,\varepsilon}\right)$.
For $k\geq 2$,
\begin{align*}
&
\left( \mathbb E^{\tilde \nu} - \mathbb E^\nu \right)Z_i^k =
\int_{B_{i,\varepsilon}} \int_0^1 k(z_i+t\alpha)^{k-1} \alpha \, dt\, d \nu 
\\ = &
\alpha \int_{B_{i,\varepsilon}} \int_0^1 k\left( \hat z_i^{k-1} + \left( (z_i+t\alpha)^{k-1} - z_i^{k-1} \right) + \left( z_i^{k-1} -\hat z_i^{k-1} \right)  \right) \, dt \, d \nu
\\ = &
\alpha \int_{B_{i,\varepsilon}} \int_0^1 k\left( \hat z_i^{k-1} + \left( (k-1)z_i^{k-2} t\alpha+ o(\alpha) \right) + \left( (k-1)\hat z_i^{k-2}(z_i-\hat z_i) + o(|z_i-\hat z_i|) \right)  \right) \, dt \, d \nu
.
\end{align*}
Since $z_i \in [\hat z_i - \varepsilon, \hat z_i + \varepsilon ]$, this is
\begin{align*}
	&
	\left( \mathbb E^{\tilde \nu} - \mathbb E^\nu \right)Z_i^k =
	\alpha \int_{B_{i,\varepsilon}} \int_0^1 k\left( \hat z_i^{k-1} + O(|\alpha|) + O(\varepsilon) \right) \, dt \, d \nu
	.
\end{align*}
\end{proof}

\begin{lemma}
\label{L premium perturbation}
Let $P_i$ be the premium calculation principle \eqref{Eq premium moments relaxed}, and suppose that $\Psi_i$ is continuously differentiable in $D_{k_i}$.
Then, for every $\nu \in \overline{\mathcal H}_X$ such that $\nu_{Z_i}$ is not concentrated at one single point:
\begin{align}
& \label{Eq perturbation of premium}
P_i(\tilde \nu) - P_i(\nu) =
\alpha \nu\left( B_{i,\varepsilon} \right) \left(
\sum_{j=1}^{k_i}
j \hat z_i^{j-1} \frac{\partial \Psi_i }{\partial u_j} + O \left( |\alpha| \right) + O(\varepsilon)
\right),
\end{align}
when $\alpha \to 0$
and
$\varepsilon \to 0^+$.
Here, the partial derivatives $\frac{\partial \Psi_i }{\partial u_j}$ are evaluated at the point $u=\left( \mathbb E^\nu Z_i, \mathbb E^\nu Z_i^2 , \ldots , \mathbb E^\nu Z_i^{k_i} \right)$.

If $\nabla \Psi_i: D_{k_i} \mapsto \mathbb R^{k_i}$ can be extended by continuity to the set
\[
\left\{ x \in [0,+\infty[^{k_i}: x_1 = x_2^{\frac 1 2} = \ldots = x_{k_i}^{\frac {1}{k_i}} \right\} ,
\]
then \eqref{Eq perturbation of premium} holds for every $\nu \in \overline{\mathcal H}_X$.
\end{lemma}

\begin{proof}
Due to Lemma \ref{L location of moments}, we have
\begin{align*}
&
P_i(\tilde \nu) - P_i(\nu) = 
\\ = &
\int_0^1 \frac d{dt} \Psi_i \left( \mathbb E^\nu Z_i+ t( \mathbb E^{\tilde \nu} - \mathbb E^\nu)  Z_i , \ldots , \mathbb E^\nu Z_i^{k_i} + t( \mathbb E^{\tilde \nu} - \mathbb E^\nu)  Z_i^{k_i} \right) dt .
\end{align*}
Hence, the Lemma follows from Lemma \ref{L perturbation of moments}.
\end{proof}

Now, we can formulate the main result in this section:

\begin{theorem}
\label{T optimality condition relaxed}
Let $U:]-\infty,c] \mapsto \mathbb R$ be continuously differentiable in $]-\infty,c[$, let
$\nu \in \overline{\mathcal H}_X$ be an optimal randomized strategy, and suppose that $\mathbb E^\nu U(L_\nu) > -\infty$.
Fix $i \in \{1,2, \ldots , n\}$ such that the function $\Psi_i$ is
continuous in $D_{k_i} \cup \left\{ u \in [0,+\infty[^{k_i}: u_i = u_2^{\frac 1 2} = \ldots = u_{k_i}^{\frac 1{k_i}} \right\}$, continuously differentiable in $D_{k_i}$, and suppose the following assumptions hold:
\begin{description}
\item[{\bf (A1)}]
Either the gradient $\nabla \Psi_i: D_{k_i} \mapsto \mathbb R^{k_i}$ can be extended by continuity to the set 
\linebreak
$\left\{ u \in [0,+\infty[^{k_i}: u_1=u_2^{\frac 12} = \ldots = u_{k_i}^{\frac 1{k_i}} \right\}$, or the marginal distribution $\nu_{Z_i}$ is not concentrated at a single point.
\item[{\bf (A2)}]
There is some $\delta >0$ such that $\mathbb E U(L_\nu - \delta ) >-\infty$.
\end{description}
Then:
\begin{enumerate}
\item \label{T optimality S1}
The inequality 
\begin{align}
& \label{Eq optimility condition a>0}
\left\lfloor \mathbb E^\nu_{\left. \left( X_{[i]}, Z_{[i]} \right) \right| X_i } U^\prime \left( L_\nu \right) \right\rfloor(\hat x_i, \hat z_i) \leq
\sum_{j=1}^{k_i} j \hat z_i^{j-1} \frac{\partial \Psi_i}{\partial u_j}
\mathbb E^\nu U^\prime (L_\nu) , 
\end{align}
holds for any $(\hat x_i, \hat z_i)\in \mathbb R^2$ such that 
\begin{align}
& \label{Eq z<x}
0 \leq \hat z_i < \hat x_i, \qquad
\nu (B_{i,\varepsilon}) >0 \ \ \forall \varepsilon>0.
\end{align}
\item \label{T optimality S2}
The inequality 
\begin{align}
& \label{Eq optimility condition a<0}
\left\lceil \mathbb E^\nu_{\left. \left( X_{[i]}, Z_{[i]} \right) \right| X_i } U^\prime ( L_\nu ) \right\rceil(\hat x_i, \hat z_i) \geq
\sum_{j=1}^{k_i} j \hat z_i^{j-1} \frac{\partial \Psi_i}{\partial u_j}
\mathbb E^\nu U^\prime (L_\nu) ,
\end{align}
holds for any $(\hat x_i, \hat z_i)\in \mathbb R^2$ such that 
\begin{align}
& \label{Eq z>0}
0 < \hat z_i \leq \hat x_i, \qquad
\nu (B_{i,\varepsilon}) >0 \ \ \forall \varepsilon>0.
\end{align}
\end{enumerate}
Here, $\lfloor \varphi \rfloor$ and $\lceil \varphi \rceil$ denote, respectively, the lower semicontinuous and the upper semicontinuous envelopes of the function $\varphi$. 
The partial derivatives $\frac{\partial \Psi_i }{\partial u_j}$ are evaluated at the point $u=\left( \mathbb E^\nu Z_i, \mathbb E^\nu Z_i^2 , \ldots , \mathbb E^\nu Z_i^{k_i} \right)$.

Under assumption {\bf (A1)} alone (i.e., {\bf (A2)} may fail):
\begin{enumerate}[resume]
\item \label{T optimality S3}
If there is some $(\hat x_i, \hat z_i) \in \mathbb R^2$ satisfying \eqref{Eq z<x}, such that $\sum\limits_{j=1}^{k_i} j \hat z_i^{j-1} \frac{\partial \Psi_i}{\partial u_j} < 0$, then
\linebreak
$\nu \left\{ U(L_\nu) = U(c) \right\} = 1 $.
\item \label{T optimality S4}
Inequality \eqref{Eq optimility condition a<0}
holds for every $(\hat x_i, \hat z_i)\in \mathbb R^2$ satisfying \eqref{Eq z>0}
such that $\sum\limits_{j=1}^{k_i} j \hat z_i^{j-1} \frac{\partial \Psi_i}{\partial u_j} >0$.
\end{enumerate}
\end{theorem}

\begin{remark}
\label{Rmk 1}
From the definition of $\mathcal H_X$, it follows that if $\nu_{Z_i}$ is concentrated in a single point $a$, then $\Pr \{ X_i < a \} = 0 $.
To see this, notice that
$\nu_{Z_i} \{ a \}= 1$ implies
\begin{align*}
\Pr \{ X_i < a \} = &
\nu_{(X_i,Z_i)} \{ (x_i,z_i): x_i < a\} =
\nu_{(X_i,Z_i)} \{ (x_i,z_i): 0 \leq z_i \leq x_i < a\}
\\ \leq &
\nu_{(X_i,Z_i)} \{ (x_i,z_i): z_i < a\} = 0 .
\end{align*}
In particular, if $\Pr \{X_i<a\}>0$ holds for every $a>0$ and $\nu_{Z_i}$ is concentrated in a single point, then that point must be $0$.
\end{remark}

\begin{remark}
\label{Rmk 2}
If $U$ is a exponential utility function, then assumption {\bf (A2)} holds for any
$\nu \in \overline{\mathcal H}_X$
such that $\mathbb E^\nu U(L_\nu) > - \infty$.
Therefore, we may take assumption {\bf (A2)} as granted whenever there is some constant $R>0$ such that
\[
\liminf_{x \to - \infty} \left( e^{Rx} U(x) \right) > -\infty .
\]
\end{remark}

\begin{remark}
Notice that $\mathbb E^\nu_{\left. \left( X_{[i]}, Z_{[i]} \right) \right| X_i } U^\prime ( L_\nu ) (\hat x_i, \hat z_i)$ is the (infinitesimal) variation of the expected utility due to an (infinitesimal) variation in the cover of risk $i$ in the neighbourhood of $\hat x_i$, while $\sum\limits_{j=1}^{k_i} j \hat z_i^{j-1} \frac{\partial \Psi_i}{\partial u_j}$ is the corresponding variation in the premium for risk $i$.
Thus, conditions \eqref{Eq optimility condition a>0}, \eqref{Eq optimility condition a<0} express a very natural economic trade-off between the (local) effect of changing the cover for the event of a particular level of risk $i$ and the (global) effect on the expected utility due to the corresponding change in premium amount.
\end{remark}

\begin{proof}[Proof of Theorem \ref{T optimality condition relaxed}]
Notice that, for any $\nu \in \overline{ \mathcal H}_X$, $\alpha \in \mathbb R$
and
$\varepsilon \geq 0$:
\begin{align}
\mathbb E^{\tilde \nu} U\left( L_{\tilde \nu} \right) = & \notag
\mathbb E^{\tilde \nu} U\left( L_\nu - \left( P_i(\tilde \nu) - P_i(\nu) \right) \right) 
\\ = & \notag
\int_{B_{i,\varepsilon}} U\left( L_\nu - \left( P_i(\tilde \nu) - P_i(\nu) \right) + \alpha \right) d \nu +
\\ & + 
\label{Eq perturbed utility}
\int_{\mathbb R^{2n} \setminus B_{i,\varepsilon}} U\left( L_\nu - \left( P_i(\tilde \nu) - P_i(\nu) \right) \right) d \nu .
\end{align}

Let $\nu \in \overline{\mathcal H}_X$ be an optimal randomized strategy and suppose that assumption {\bf (A1)} holds.

Suppose there is some $(\hat x_i, \hat z_i)$ satisfying \eqref{Eq z<x} such that $\sum\limits_{j=1}^{k_i} j \hat z_i^{j-1} \frac{\partial \Psi_i}{\partial u_j} < 0$.
By Lemma \ref{L premium perturbation}, $P_i(\tilde \nu) - P_i(\nu) <0$ holds for sufficiently small $\varepsilon>0$, $\alpha >0$.
Therefore,
\begin{align*}
\left( L_\nu - \left( P_i(\tilde \nu) - P_i(\nu) \right) + \alpha \right) \chi_{B_{i,\varepsilon}} +
\left( L_\nu - \left( P_i(\tilde \nu) - P_i(\nu) \right) \right) \chi_{ B_{i,\varepsilon}^c} 
>
L_\nu.
\end{align*}
Thus, 
statement \ref{T optimality S3} must hold, due to optimality of $\nu$ and 
monotonicity of $U$.

Now, suppose that assumption {\bf (A2)} also holds, and fix $\delta >0$ such that $\mathbb E^\nu U \left( L_\nu - \delta \right) >- \infty$.
Due to Lemma \ref{L premium perturbation}, $P_i(\tilde \nu) - P_i(\nu) < \delta$ whenever $\varepsilon $ and $\alpha $ are sufficiently small.
Further, optimality  of $\nu$ and equality \eqref{Eq perturbed utility} imply
\begin{align}
& \notag
\mathbb E^{\tilde \nu}U(L_{\tilde \nu} ) \leq \mathbb E ^\nu U(L_\nu) \Leftrightarrow
\\ \Leftrightarrow & \notag
\int_{B_{i,\varepsilon}} 
U\left( L_\nu - \left( P_i(\tilde \nu) - P_i(\nu) \right) + \alpha \right) - U\left( L_\nu - \left( P_i(\tilde \nu) - P_i(\nu) \right) \right) d \nu \leq
\\ & \leq \notag
\mathbb E^\nu \left( U\left( L_\nu \right) - U\left( L_\nu - \left( P_i(\tilde \nu) - P_i(\nu) \right) \right) \right) 
\\ \Leftrightarrow & \notag
\alpha
\int_{B_{i,\varepsilon}} \int_0^1 U^\prime \left( L_\nu - \left( P_i(\tilde \nu) - P_i(\nu) \right) + t\alpha \right) 
dt \, d \nu \leq
\\ & \leq \label{Eq inequality optimality}
\left( P_i(\tilde \nu) - P_i(\nu) \right)
\mathbb E^\nu
\int_0^1 U^\prime \left( L_\nu - t\left( P_i(\tilde \nu) - P_i(\nu) \right) \right) dt 
.
\end{align}
If $(\hat x_i, \hat z_i)$ satisfy \eqref{Eq z<x} and $\alpha >0$, then dividing
both sides of \eqref{Eq inequality optimality} by $\alpha \nu(B_{i,\varepsilon})$, 
using Fubini's theorem and conditional independence of $\left( X_{[i]},Z_{[i]} \right)$ and $Z_i$ given $X_i$,
we obtain
\begin{align}
& \notag
\frac 1{\nu (B_{i,\varepsilon})}
\int_0^1
\int_{B_{i,\varepsilon}} \mathbb E^\nu_{\left. \left( X_{[i]},Z_{[i]} \right) \right| X_i}
U^\prime \left( L_\nu \right) \left( x_i, z_i - \left( P_i(\tilde \nu) - P_i(\nu) \right) + \alpha \right) d \nu \leq
\\ \leq & \label{Eq inequality a>0}
\frac{P_i(\tilde \nu) - P_i(\nu) }{\alpha \nu(B_{i,\varepsilon})}
\mathbb E^\nu \int_0^1 U^\prime \left( L_\nu - t\left( P_i(\tilde \nu) - P_i(\nu) \right) \right) d t .
\end{align}
This implies
\begin{align*}
& 
\inf_{|x_i-\hat x_i| + |z_i - \hat z_i | \leq  \varepsilon + \alpha + P_i(\tilde \nu)-P_i(\nu)} 
\mathbb E^\nu_{\left. \left( X_{[i]},Z_{[i]} \right) \right| X_i}
U^\prime \left( L_\nu \right) \left( x_i, z_i \right) \leq
\\ \leq & 
\frac{P_i(\tilde \nu) - P_i(\nu) }{\alpha \nu(B_{i,\varepsilon})}
\mathbb E^\nu \int_0^1 U^\prime \left( L_\nu - t\left( P_i(\tilde \nu) - P_i(\nu) \right) \right) d t .
\end{align*}
Therefore, taking successive limits when $\varepsilon \to 0^+$ and $\alpha \to 0^+$ 
and taking into account Lemma \ref{L premium perturbation},
one obtains \eqref{Eq optimility condition a>0}.

Under assumption {\bf (A2)}, inequality \eqref{Eq optimility condition a<0} can be proved by a similar argument. Just notice that for $\alpha <0$ inequality in \eqref{Eq inequality a>0} holds with the opposite inequality sign, and use supremum instead of infimum in the last step.

To prove statement \ref{T optimality S4}, notice that if $(\hat x_i, \hat z_i)$ satisfies \eqref{Eq z>0} and 
$\sum\limits_{j=1}^{k_i} j \hat z_i^{j-1} \frac{\partial \Psi_i}{\partial u_j} >0$,
then $P_i(\tilde \nu) - P_i(\nu) <0$ for $\varepsilon >0$ and $\alpha <0$ close to zero.
Hence, integrability of $U\left(L_\nu - \left( P_i(\tilde \nu) - P_i(\nu) \right) \right) $ follows from monotonicity of $U$, and therefore inequality \eqref{Eq optimility condition a<0} does not depend on assumption {\bf (A2)}.
\end{proof}

In Section \ref{S deterministic}, we will present conditions guaranteeing  that optimal treaties are deterministic.
In view of those results, we present the following version of Theorem 
\ref{T optimality condition relaxed}.

\begin{theorem}
\label{T optimality condition deterministic}
Suppose that $U$ is continuously differentiable in $]-\infty,c[$, and let $\nu \in \overline{\mathcal H}_X$ be an optimal randomized strategy.
Fix $i \in \{1,2, \ldots , n\}$ such that the function $\Psi_i$ is
continuous in $D_{k_i} \cup \left\{ u \in [0,+\infty[^{k_i}: u_1 = u_2^{\frac 1 2} = \ldots = u_{k_i}^{\frac 1{k_i}} \right\}$, continuously differentiable in $D_{k_i}$, and suppose that assumptions {\bf (A1)}, {\bf (A2)} of Theorem \ref{T optimality condition relaxed} hold.

If the marginal distributon $\mu_{X_i}$ is absolutely continuous except possibly for an atom at $x_i=0$ and the strategy prescribed by $\nu$ for the risk $i$ is deterministic (i.e., $Z_i = Z_i(X_i)$), then the optimal treaty $Z_i $ satisfies the conditions
\begin{align*}
&
\mathbb E^\nu_{\left. \left( X_{[i]}, Z_{[i]} \right) \right| X_i } U^\prime \left( L_\nu \right) \left( x_i,Z_i(x_i) \right) \leq 
\sum_{j=1}^{k_i} \frac{\partial \Psi_i}{\partial u_j} j Z_i(x_i)^{j-1} \mathbb E^\nu U^\prime \left( L_\nu \right) 
\\ & \hspace{40mm}
\text{for } \mu_{X_i}\text{-a.e. } x_i>0 \ \text{such that }Z_i(x_i) < x_i;
\\ &
\mathbb E^\nu_{\left. \left( X_{[i]}, Z_{[i]} \right) \right| X_i } U^\prime \left( L_\nu \right) \left( x_i, Z_i(x_i) \right) \geq 
 \sum_{j=1}^{k_i} \frac{\partial \Psi_i}{\partial u_j} j Z_i(x_i)^{j-1} \mathbb E^\nu U^\prime \left( L_\nu \right)
\\ & \hspace{40mm}
\text{for } \mu_{X_i}\text{-a.e. } x_i>0 \ \text{such that }Z_i(x_i) >0.
\end{align*}
The partial derivatives $\frac{\partial \Psi_i}{\partial u_j}$ are evaluated at the point 
\[
u = \left( \mathbb E Z_i(X_i), \mathbb E Z_i(X_i)^2, \ldots , \mathbb E Z_i(X_i)^{k_i} \right).
\]
\end{theorem}

\begin{proof}
The proof follows an argument similar to the proof of Theorem \ref{T optimality condition relaxed}, with some adaptations.

The absolute continuity assumption on the marginal distribution $\mu_{X_i}$ means that there is a constant $\alpha \in [0,1[$ and a non-negative function $f_{X_i}$ such that
\[
\mu_{X_i}(A) = \alpha \delta_0(A) + \int_A f_{X_i}(x_i) dx_i \qquad \forall A \in \mathcal B_{[0,+\infty[},
\]
where $\delta_0$ denotes the Dirac measure concentrated at $x_i=0$.

Let $Z_i$ be the optimal treaty for risk $i$, and 
$\nu \in \overline{\mathcal H}_X$ 
be the optimal strategy for the total portfolio of risks.
Thus,
\begin{align*}
\nu(A) = \int_{\mathbb R} \int_{\mathbb R^{2n-2}} \chi_A\left( x,z_{[i]},Z_i(x_i) \right) \nu_{\left. \left( X_{[i]}, Z_{[i]} \right) \right| X_i }\left( x_i, d(x_{[i]},z_{[i]}) \right) \mu_{X_i}(dx_i)
\end{align*}
for any $A \in \mathcal B_{\mathbb R^{2n}}$.
For fixed $\hat x_i \in ]0,+\infty[$, $\varepsilon >0$, $\zeta \in \mathbb Q \cap [0, \hat x_i[$, let
\[
\tilde Z_i(x_i) = \left\{ \begin{array}{ll}
Z_i(x_i),& \text{for } x_i \notin ]\hat x_i - \varepsilon, \hat x_i + \varepsilon [, \\
\zeta,& \text{for } x_i \in ]\hat x_i - \varepsilon, \hat x_i + \varepsilon [, 
\end{array} \right.
\]
and let $\tilde \nu$ be the corresponding measure in $\mathcal B_{\mathbb R^{2n}}$.
The argument used to prove Lemma \ref{L perturbation} shows that $\tilde \nu \in \overline{\mathcal H}_X$, provided $\varepsilon$ is sufficiently small.

An argument similar to the proof of Lemma \ref{L premium perturbation} shows that
\begin{align*}
P_i(\tilde \nu) - P_i(\nu) = &
2 \varepsilon \left(\zeta - Z_i(\hat x_i) \right) \sum_{j=1}^{k_i} \frac{\partial \Psi_i}{\partial u_j} j Z_i(\hat x_i)^{j-1} f_{X_i}(\hat x_i) +
\varepsilon o \left( \zeta - Z_i(\hat x_i) \right) + o(\varepsilon ),
\end{align*}
for every $\hat x_i$, a Lebesgue point of the functions $Z_i^j f_{X_i}$, $j=0,1, \ldots, k_i$.
Thus,
\begin{align*}
&
\mathbb E^{\tilde \nu} U\left( L_{\tilde \nu} \right) =
\int_{\mathbb R \setminus ]\hat x_i - \varepsilon, \hat x_i + \varepsilon[} \int_{\mathbb R^{2n-2}} U \left( L_\nu - \left( P_i(\tilde \nu ) - P_i(\nu) \right) \right) d \nu_{ \left. \left( X_{[i]},Z_{[i]} \right) \right| X_i} d\mu_{X_i} +
\\ & +
\int_{\hat x_i - \varepsilon}^{\hat x_i + \varepsilon} \int_{\mathbb R^{2n-2}} U \left( L_\nu - \left( P_i(\tilde \nu ) - P_i(\nu) \right) + \zeta - Z_i \right) d\nu_{ \left. \left( X_{[i]},Z_{[i]} \right) \right| X_i} d\mu_{X_i} 
\\ = &
\mathbb E^\nu U \left( L_\nu - \left( P_i(\tilde \nu ) - P_i(\nu) \right) \right) +
\\ & +
\int_{\hat x_i - \varepsilon}^{\hat x_i + \varepsilon} \int_{\mathbb R^{2n-2}} \left( U \left( L_\nu - \left( P_i(\tilde \nu ) - P_i(\nu) \right) + \zeta - Z_i \right) - U \left( L_\nu - \left( P_i(\tilde \nu ) - P_i(\nu) \right) \right) \right)
\\ & \hspace{30mm} 
d \nu_{ \left. \left( X_{[i]},Z_{[i]} \right) \right| X_i} d\mu_{X_i} .
\end{align*}
Optimality of $\nu$ implies that $\mathbb E^{\tilde \nu} U\left( L_{\tilde \nu} \right) \leq \mathbb E^\nu U \left( L_\nu \right)$, that is
\begin{align*}
&
\int_{\hat x_i - \varepsilon}^{\hat x_i + \varepsilon} \int_{\mathbb R^{2n-2}} \left( U \left( L_\nu - \left( P_i(\tilde \nu ) - P_i(\nu) \right) + \zeta - Z_i \right) - U \left( L_\nu - \left( P_i(\tilde \nu ) - P_i(\nu) \right) \right) \right)
\\ & \hspace{30mm} 
d \nu_{ \left. \left( X_{[i]},Z_{[i]} \right) \right| X_i} d\mu_{X_i} \leq 
\\ \leq &
\mathbb E^\nu \left( U \left( L_\nu - \left( P_i(\tilde \nu) - P_i(\nu) \right) \right) -
U \left( L_\nu \right) \right) .
\end{align*}
Since $U$ is concave and $P_i(\tilde \nu) \to P_i(\nu)$ when $\varepsilon \to 0$, it follows that for any constant $\eta\in \mathbb Q \cap ]0,+\infty[$, the inequality
\begin{align*}
&
\int_{\hat x_i - \varepsilon}^{\hat x_i + \varepsilon} \mathbb E^\nu_{ \left. \left( X_{[i]},Z_{[i]} \right) \right| X_i} \int_0^1 U^\prime \left( L_\nu + \eta + t(\zeta - Z_i) \right) dt \left( \zeta - Z_i \right) d\mu_{X_i} \leq 
\\ \leq &
\left( P_i(\tilde \nu) - P_i(\nu) \right) \mathbb E^\nu \int_0^1 U^\prime \left( L_\nu - t\left( P_i(\tilde \nu) - P_i(\nu) \right) \right) dt 
\end{align*}
holds for every sufficiently small $\varepsilon >0$.
Dividing both sides by $\varepsilon$, making $\varepsilon \to 0$, and using Lebesgue's dominated convergence theorem, we obtain
\begin{align*}
&
\mathbb E^\nu_{ \left. \left( X_{[i]},Z_{[i]} \right) \right| X_i} \int_0^1 U^\prime \left( L_\nu + \eta + t(\zeta - Z_i(\hat x_i)) \right) dt \left( \zeta - Z_i(\hat x_i) \right) \leq
\\ \leq &
(\zeta - Z_i(\hat x_i)) \sum_{j=1}^{k_i} \frac{\partial \Psi_i}{\partial u_j}j Z_i(\hat x_i)^{j-1} \mathbb E^\nu U^\prime \left( L_\nu \right) +
o \left( \zeta - Z_i(\hat x_i) \right) ,
\end{align*}
for every $\hat x_i>0$, a Lebesgue point of the functions $Z_i^j f_{X_i}$, $j=0,1,\ldots, k_i$ and 
\linebreak
$\mathbb E^\nu_{ \left. \left( X_{[i]},Z_{[i]} \right) \right| X_i} \int_0^1 U^\prime \left( L_\nu + \eta + t(\zeta - Z_i) \right) dt \left( \zeta - Z_i \right) f_{X_i}$, such that $f_{X_i}(\hat x_i)>0$.

If $Z_i(\hat x_i) < \hat x_i$ (resp., $Z_i(\hat x_i)>0$), then we can pick a sequence $\zeta_k \searrow Z_i(\hat x_i)$ (resp., $\zeta_k \nearrow Z_i(\hat x_i)$) and obtain
\begin{align*}
&
\mathbb E^\nu_{ \left. \left( X_{[i]},Z_{[i]} \right) \right| X_i} U^\prime \left( L_\nu + \eta \right) \leq
\sum_{j=1}^{k_i} \frac{\partial \Psi_i}{\partial u_j}j Z_i(\hat x_i)^{j-1} \mathbb E^\nu U^\prime \left( L_\nu \right) 
\end{align*}
(resp., $\mathbb E^\nu_{ \left. \left( X_{[i]},Z_{[i]} \right) \right| X_i} U^\prime \left( L_\nu + \eta \right) \leq
\sum\limits_{j=1}^{k_i} \frac{\partial \Psi_i}{\partial u_j}j Z_i(\hat x_i)^{j-1} \mathbb E^\nu U^\prime \left( L_\nu \right) $).
Since $\eta $ is arbitrary, the result follows by Lebesgue's dominated convergence theorem.
\end{proof}

\begin{remark}
\label{Rmk Centeno08}
If the utility function is exponential and $X_i$ is independent of $X_{[i]}$ (or $n=1$), then the optimality conditions in Theorem \ref{T optimality condition deterministic} coincide with the conditions in
\linebreak
\cite[Theorem 3]{GC08}.
\end{remark}

\section{Deterministic optimal treaties}\label{S deterministic}

In this section we show that under mild conditions the optimal treaty is of the deterministic type
(Subsection \ref{SS deterministic solutions}).
The remaining of this section contains some particular consequences of these results. In Subsections \ref{SS expected value} and \ref{SS variance-related}, we show that some important premium calculation principles (the expected value and variance-related principles, repectively) satisfy the conditions above and provide the particular form of Theorem \ref{T optimality condition deterministic} for these principles.
In Subsection \ref{SS CaiWei}, we compare our results with known results for the expected value principle \cite{CW12}.

\subsection{Problems with deterministic optimal strategies}
\label{SS deterministic solutions}

The main result in this section is the following theorem.

\begin{theorem}
\label{T deterministic treaties}
Suppose that $U$ is differentiable in $]-\infty,c[$, let $\nu \in \overline{\mathcal H}_X$ be an optimal randomized strategy, and suppose that assumption {\bf (A2)} of Theorem 
\ref{T optimality condition relaxed} holds.
Pick $i \in \{1,2, \ldots , n\}$ such that the marginal distribution $\mu_{X_i}$ has no atoms except possibly at $x_i=0$, and the function $\Psi_i$ is continuous in $D_{k_i} \cup \left\{ u \in [0,+\infty[^{k_i}: u_1 = u_2^{\frac 12} = \ldots = u_{k_i}^{\frac 1{k_i}} \right\}$, continuously differentiable in $D_{k_i}$.
Let $\varphi_{i,\nu}:[0,+\infty[ \mapsto \mathbb R$ be the function
\[
\varphi_{i,\nu} (t) =
\mathbb E^\nu U^\prime \left(L_\nu \right) \sum_{j=1}^{k_i} \frac{\partial \Psi_i}{\partial u_j}\left( \mathbb E^\nu Z_i, \mathbb E^\nu Z_i^2, \ldots, \mathbb E^\nu Z_i^{k_i} \right) t^j.
\]

If the functions $U$, $-\varphi_{i,\nu}$ are concave, with at least one of them being strictly concave, then the strategy prescribed by $\nu$ for the risk $i$ is deterministic.
\end{theorem}

\begin{proof}
Pick $\nu \in \overline{\mathcal H}_X$, and suppose that the strategy it prescribes for risk $i$ is not deterministic.
Then, the marginal distribution $\nu_{Z_i}$ is not concentrated at a single point and there are $\hat x_i, \hat z_i, \tilde z_i$, with $0 \leq \hat z_i < \tilde z_i \leq \hat x_i$, such that
\[
\nu\left(B_{i,\varepsilon}(\hat x_i, \hat z_i) \right) >0,
\quad 
\nu\left(B_{i,\varepsilon}(\hat x_i, \tilde z_i) \right) >0,
\qquad 
\forall \varepsilon >0 .
\]
Fix $\hat x_i, \hat z_i, \tilde z_i$ as above and pick constants $\alpha, \beta >0$ such that $\alpha + \beta < \tilde z_i - \hat z_i$, and therefore
\[
0 \leq \hat z_i < \hat z_i + \alpha < \tilde z_i - \beta < \tilde z_i \leq \hat x_i.
\]
Let $\hat A, \tilde A $ be open cylinders
\[
\hat A = \left\{ (x,z ) \in \mathbb R^{2n} : (x_i,z_i) \in \hat A_i \right\},
\qquad
\tilde A = \left\{ (x,z ) \in \mathbb R^{2n} : (x_i,z_i) \in \tilde A_i \right\},
\]
with
\begin{align*}
&
\hat A_i \subset \left\{ (x_i, z_i) \in \mathbb R^2: (x_i- \hat x_i)^2 + (z_i - \hat z_i)^2 < \varepsilon^2 \right\}
\\ &
\tilde A_i \subset \left\{ (x_i, z_i) \in \mathbb R^2: (x_i- \hat x_i)^2 + (z_i - \tilde z_i)^2 < \varepsilon^2 \right\} ,
\end{align*}
and let $\tilde \nu$ be the measure
\begin{align*}
\tilde \nu (B) = &
\nu \left( B \setminus (\hat A \cup \tilde A) \right) +
\nu \left\{ (x,z) \in \hat A: (x,z + \alpha e_i) \in B \right\} +
\\ & +
\nu \left\{ (x,z) \in \tilde A: (x,z - \beta e_i) \in B \right\} .
\end{align*}
The argument used to prove Lemma \ref{L perturbation} shows that,
for $\varepsilon$ sufficiently small, $\tilde \nu$ is an element of 
$\overline{\mathcal H}_X$.
The argument used to prove Lemma \ref{L perturbation of moments} shows that
\begin{align*}
&
\left( \mathbb E^{\tilde \nu} - \mathbb E^\nu \right)Z_i^k =
\nu( \hat A) \left( (\hat z_i + \alpha)^k - \hat z_i^k + O(\varepsilon) \right) +
\nu( \tilde A) \left( (\tilde z_i - \beta)^k - \tilde z_i^k + O(\varepsilon) \right),
\end{align*}
for any $k \in \mathbb N$ such that $\mathbb EZ_i^k < + \infty$.
Therefore, the argument used to prove Lemma \ref{L premium perturbation} shows that
\begin{align}
P_i(\tilde \nu ) - P_i(\nu) = & \notag
\nu( \hat A) \sum_{j=1}^{k_i} \frac{\partial \Psi_i}{\partial u_j} \times \left( (\hat z_i + \alpha)^j - \hat z_i^j + O(\varepsilon) \right) +
\\ & + \notag
\nu( \tilde A) \sum_{j=1}^{k_i} \frac{\partial \Psi_i}{\partial u_j} \times \left( (\tilde z_i - \beta)^j - \tilde z_i^j + O(\varepsilon) \right) +
\\ & + \label{Z7}
o \left( \nu (\hat A) + \nu(\tilde A) \right) ,
\end{align}
where the partial derivatives $\frac{\partial \Psi_i}{\partial u_j}$ are evaluated at the point $u = \left( \mathbb E^\nu Z_i,\mathbb E^\nu Z_i^2, \ldots , \mathbb E^\nu Z_i^{k_i} \right)$.
It follows that
\begin{align*}
&
\mathbb E^{\tilde \nu} U\left( L_{\tilde \nu} \right) - \mathbb E^\nu U \left( L_\nu \right) =
\\ = &
\int_{\hat A} U \left( L_\nu - (P_i(\tilde \nu) - P_i(\nu)) + \alpha \right) d\nu +
\int_{\tilde A} U \left( L_\nu - (P_i(\tilde \nu) - P_i(\nu) ) - \beta \right) d\nu +
\\ & +
\int_{\mathbb R^{2n} \setminus (\hat A \cup \tilde A)} U \left( L_\nu - (P_i(\tilde \nu) - P_i(\nu) ) \right) d\nu -
\int_{\mathbb R^{2n} } U \left( L_\nu \right) d\nu .
\end{align*}
Under assumption {\bf (A2)} and differentiability of $U$, this is
\begin{align*}
&
\mathbb E^{\tilde \nu} U\left( L_{\tilde \nu} \right) - \mathbb E^\nu U \left( L_\nu \right) =
\\ = &
\alpha \int_{\hat A} \int_0^1 U^\prime \left( L_\nu - (P_i(\tilde \nu) - P_i(\nu)) + t \alpha \right) dt\, d\nu -
\\ & -
\beta \int_{\tilde A} \int_0^1 U^\prime \left( L_\nu - (P_i(\tilde \nu) - P_i(\nu) ) - t\beta \right) dt\, d\nu -
\\ & -
(P_i(\tilde \nu) - P_i(\nu) )\int_{\mathbb R^{2n} } \int_0^1 U^\prime \left( L_\nu - t(P_i(\tilde \nu) - P_i(\nu) ) \right) dt \, d\nu .
\end{align*}
Taking into account concavity of $U$ and estimate \eqref{Z7}, this implies
\begin{align*}
&
\mathbb E^{\tilde \nu} U\left( L_{\tilde \nu} \right) - \mathbb E^\nu U \left( L_\nu \right) \geq
\\ \geq &
\alpha \int_{\hat A} U^\prime \left( L_\nu + \alpha + O(\nu(\hat A)+ \nu(\tilde A) ) \right) d\nu -
\\ & -
\beta \int_{\tilde A} U^\prime \left( L_\nu - \beta + O(\nu(\hat A)+ \nu(\tilde A) ) \right) d\nu -
\\ & -
(P_i(\tilde \nu) - P_i(\nu) )\int_{\mathbb R^{2n} } U^\prime \left( L_\nu + O(\nu(\hat A)+ \nu(\tilde A) ) \right) d\nu 
\\ \geq &
\alpha \int_{\hat A} \mathbb E^\nu_{\left. \left(X_{[i]},Z_{[i]}\right) \right|X_i}U^\prime \left( L_\nu \right) \left( \hat x_i, \hat z_i + \alpha + O( \varepsilon ) \right) d\nu -
\\ & -
\beta \int_{\tilde A} \mathbb E^\nu_{\left. \left(X_{[i]},Z_{[i]}\right) \right|X_i}U^\prime \left( L_\nu \right) \left( \hat x_i, \tilde z_i - \beta + O( \varepsilon ) \right) d\nu -
\\ & -
(P_i(\tilde \nu) - P_i(\nu) )\int_{\mathbb R^{2n} } U^\prime \left( L_\nu + O(\nu(\hat A)+ \nu(\tilde A) ) \right) d\nu 
\\ = &
\alpha \nu(\hat A) E^\nu_{\left. \left(X_{[i]},Z_{[i]}\right) \right|X_i}U^\prime \left( L_\nu \right) \left( \hat x_i, \hat z_i + \alpha + O( \varepsilon ) \right) -
\\ & -
\beta \nu(\tilde A) \mathbb E^\nu_{\left. \left(X_{[i]},Z_{[i]}\right) \right|X_i}U^\prime \left( L_\nu \right) \left( \hat x_i, \tilde z_i - \beta + O( \varepsilon ) \right) -
\\ & -
\nu(\hat A) \sum_{j=1}^{k_i} \frac{\partial \Psi_i}{\partial u_j} \left((\hat z_i + \alpha )^j - \hat z_i^j + O(\varepsilon) \right) \mathbb E^\nu U^\prime \left(L_\nu \right) -
\\ & -
\nu(\tilde A) \sum_{j=1}^{k_i} \frac{\partial \Psi_i}{\partial u_j} \left((\tilde z_i - \beta )^j - \tilde z_i^j + O(\varepsilon) \right) \mathbb E^\nu U^\prime \left(L_\nu \right) +
o\left( \nu (\hat A) + \nu(\tilde A) \right)
.
\end{align*}
Since the distribution of $X_i$ has no atoms except possibly $0$ and $\hat x_i>0$, for every $\varepsilon >0$ there is a pair of cylinders $\hat A_\varepsilon$, $\tilde A_\varepsilon$ such that:
\begin{enumerate}
\item 
$\hat A_\varepsilon = B_{i,\varepsilon}(\hat x_i, \hat z_i), 
\qquad \text{or} \qquad
\tilde A_\varepsilon = B_{i,\varepsilon}(\hat x_i, \tilde z_i)$;
\item
$\alpha \nu \left( \hat A_\varepsilon \right) =
\beta \nu \left( \tilde A_\varepsilon \right) $.
\end{enumerate}
For such pairs of cylinders, the estimate above becomes
\begin{align*}
&
\mathbb E^{\tilde \nu} U\left( L_{\tilde \nu} \right) - \mathbb E^\nu U \left( L_\nu \right) \geq
\\ \geq &
\alpha \nu\left(\hat A_\varepsilon \right) \Bigg( E^\nu_{\left. \left(X_{[i]},Z_{[i]}\right) \right|X_i}U^\prime \left( L_\nu \right) \left( \hat x_i, \hat z_i + \alpha + O( \varepsilon ) \right) -
\\ & \hspace{19mm} -
\mathbb E^\nu_{\left. \left(X_{[i]},Z_{[i]}\right) \right|X_i}U^\prime \left( L_\nu \right) \left( \hat x_i, \tilde z_i - \beta + O( \varepsilon ) \right) +
\\ & \hspace{19mm} +
\mathbb E^\nu U^\prime \left(L_\nu \right) \sum_{j=1}^{k_i} \frac{\partial \Psi_i}{\partial u_j} \left(\frac{(\tilde z_i - \beta )^j - \tilde z_i^j}{-\beta} + O(\varepsilon) \right) -
\\ & \hspace{19mm} -
\mathbb E^\nu U^\prime \left(L_\nu \right) \sum_{j=1}^{k_i} \frac{\partial \Psi_i}{\partial u_j} \left(\frac{(\hat z_i + \alpha )^j - \hat z_i^j}\alpha + O(\varepsilon) \right) \Bigg) +
o\left( \nu \left(\hat A_\varepsilon \right)\right) 
\\ = &
\alpha \nu\left(\hat A_\varepsilon \right) \Bigg( E^\nu_{\left. \left(X_{[i]},Z_{[i]}\right) \right|X_i}U^\prime \left( L_\nu \right) \left( \hat x_i, \hat z_i + \alpha + O( \varepsilon ) \right) -
\\ & \hspace{19mm} -
\mathbb E^\nu_{\left. \left(X_{[i]},Z_{[i]}\right) \right|X_i}U^\prime \left( L_\nu \right) \left( \hat x_i, \tilde z_i - \beta + O( \varepsilon ) \right) +
\\ & \hspace{19mm} +
\frac{\varphi_{i,\nu}(\tilde z_i -\beta) - \varphi_{i,\nu}(\tilde z_i)}{-\beta} -
\frac{\varphi_{i,\nu}(\hat z_i +\alpha ) - \varphi_{i,\nu}(\hat z_i)} \alpha + O(\varepsilon)
\Bigg) +
\\ & +
o\left( \nu \left(\hat A_\varepsilon \right)\right) .
\end{align*}
Since $U$ is concave, 
\begin{align*}
&
E^\nu_{\left. \left(X_{[i]},Z_{[i]}\right) \right|X_i}U^\prime \left( L_\nu \right) \left( \hat x_i, \hat z_i + \alpha + O( \varepsilon ) \right) \geq
\mathbb E^\nu_{\left. \left(X_{[i]},Z_{[i]}\right) \right|X_i}U^\prime \left( L_\nu \right) \left( \hat x_i, \tilde z_i - \beta + O( \varepsilon ) \right) 
\end{align*}
holds for sufficiently small $\varepsilon$, with strict inequality if $U$ is strictly concave.
In addition, since $\varphi_{i,\nu}$ is convex,
\begin{align*}
\frac{\varphi_{i,\nu}(\tilde z_i -\beta) - \varphi_{i,\nu}(\tilde z_i)}{-\beta} \geq
\frac{\varphi_{i,\nu}(\hat z_i +\alpha ) - \varphi_{i,\nu}(\hat z_i)} \alpha
\end{align*}
holds, and the inequality is strict if $\varphi_{i,\nu}$ is strictly convex.

It follows that for sufficiently small $\varepsilon>0$, $\mathbb E^{\tilde \nu} U\left( L_{\tilde \nu} \right) > \mathbb E^\nu U\left( L_\nu \right) $, and therefore $\nu$ is not optimal.
\end{proof}

The following corollary gives conditions that are easier to check than the conditions in Theorem \ref{T deterministic treaties}.

\begin{corollary}
\label{C higher moments}
Suppose that $U$ is differentiable and concave in $]-\infty,c[$, and let $\nu \in \overline{\mathcal H}_X$ be an optimal randomized strategy.
Suppose that assumption {\bf (A2)} of Theorem \ref{T optimality condition relaxed} holds, and the marginal distribution $\mu_{X_i}$ has no atoms except, possibly, at $x_i=0$.

Suppose that the reinsurance premium for risk $i \in \{1,2,\ldots, n\}$ is computed by a function $\Psi_i$ of the moments of order up to $k_i \in \mathbb N$, continuous in $D_{k_i} \cup \left\{ u \in [0,+\infty[^{k_i}: u_1 = u_2^{\frac 12} = \ldots = u_{k_i}^{\frac 1{k_i}} \right\}$, continuously differentiable in $D_{k_i}$, such that
\[
\frac{\partial \Psi_i}{\partial u_j}(u) \geq 0
\qquad \forall j \geq 2, \ u \in D_{k_i} .
\]

If at least one of the following conditions holds:
\begin{enumerate}
\item
The functions $\frac{\partial \Psi_i}{\partial u_j}$, $j=1,2, \ldots, k_i$, have at most one common zero in $D_{k_i}$, and $\mathbb E^\nu U^\prime(L_\nu) >0$;
\item
$U$ is strictly concave in $]-\infty,c]$;
\end{enumerate}
then the optimal reinsurance for risk $i$ is deterministic.
\end{corollary}

\begin{proof}
Under the assumptions of the corollary, the functions $U$, $- \varphi_{i,\nu}$ are concave with at least one of them being strictly concave.
Thus, the result follows immediately from Theorem \ref{T deterministic treaties}.
\end{proof}

\subsection{The expected value principle}
\label{SS expected value}

From Theorem \ref{T optimality condition deterministic} and Corollary \ref{C higher moments}, it is easy to check that under mild conditions, the optimal stategy concerning a risk for which the reinsurance premium is computed by the expected value principle is deterministic, and derive the corresponding optimality condition.
More precisely, we obtain the following corollary:

\begin{corollary}
\label{C expected value principle}
Suppose that $U$ is continuously differentiable and concave in $]-\infty,c[$, let $\nu \in \overline{\mathcal H}_X$ be an optimal randomized strategy, and suppose that assumption {\bf (A2)} of Theorem \ref{T optimality condition relaxed} holds.

If the reinsurance premium for risk $i \in \{1,2,\ldots, n\}$ is computed by the expected value principle (i.e., $\Psi_i (u) = (1 + \theta_i)u$), and the marginal distribution $\mu_{X_i}$ is absolutely continuous except, possibly, for an atom at $x_i=0$, then the optimal reinsurance for risk $i$ is a deterministic function $Z_i \in \mathcal Z_i$, such that
\begin{align*}
&
E^\nu_{\left. \left( X_{[i]}, Z_{[i]} \right) \right| X_i } U^\prime \left( L_\nu \right) (x_i, 0) \leq
(1 + \theta_i) \mathbb E^\nu U^\prime (L_\nu) & \text{if } Z_i(x_i)=0,
\\ &
E^\nu_{\left. \left( X_{[i]}, Z_{[i]} \right) \right| X_i } U^\prime \left( L_\nu \right) (x_i, Z_i(x_i)) =
(1 + \theta_i) \mathbb E^\nu U^\prime (L_\nu) &
\text{if } 0 < Z_i(x_i) < x_i ,
\\ &
E^\nu_{\left. \left( X_{[i]}, Z_{[i]} \right) \right| X_i } U^\prime \left( L_\nu \right) (x_i, x_i) \geq
(1+ \theta_i) \mathbb E^\nu U^\prime (L_\nu) & \text{if } Z_i(x_i)=x_i,
\end{align*}
holds for $\mu_{X_i}$-almost every $x_i$.
\end{corollary}

\begin{proof}
The fact that the optimal reinsurance strategy for risk $i$ is deterministic follows immediately from Corollary \ref{C higher moments}.
Thus, the optimality conditions follow from Theorem \ref{T optimality condition deterministic}.
\end{proof}

\subsection{Stop-loss is not optimal}
\label{SS CaiWei}

It is known \cite{CW12} that if all reinsurance premia are calculated by expected value principia, and the risks are positively dependent in stochastic order, then stop-loss treaties $Z_i(x)=\max (0,x-M_i)$ are optimal among the class of all deterministic treaties such that the retained risks $X_i - Z_i(X_i)$ are increasing functions of $X_i$.
	
We will now show that stop-loss treaties are in general not optimal in the wider class of measurable deterministic treaties, even in the setting above.
Thus, the monotonicity constraint on the retained risk is typically an active constraint.
	
Consider two risks, $X_1$ and $X_2$, with absolutely continuous (marginal) distributions with support $[0,+\infty[$, each being reinsured through a stop-loss treaty, priced by an expected value premium principle:
\[
Z_i(x)=\max \left( 0, x-M_i \right) ,
\quad
P(Z_i)=(1+\theta_i) \mathbb E Z_i(X_i),
\qquad
i=1,2,
\]
and consider an exponential utility function
\[
U(x) = -e^{-Rx}.
\]
By Corollary \ref{C expected value principle}, if such a strategy is optimal, then the conditions
\[
\left\{
\begin{array}{rcll}
(1+\theta_i)E\left[e^{R(X_1-Z_1+X_2-Z_2)}\right] &\geqslant& 
E\left[ e^{R(X_1-Z_1+X_2-Z_2)}|X_i=x_i\right] & x_i \leqslant M_i \\
&& \\
(1+\theta_i)E\left[e^{R(X_1-Z_1+X_2-Z_2)}\right] &=& 
E\left[ e^{R(X_1-Z_1+X_2-Z_2)}|X_i=x_i\right] & x_i \geqslant M_i 
\end{array}
\right.
\]
hold for $\mu_{X_i}$-almost every $x_i$, $i=1,2$.
Thus, for $i=1$ (and similarly for $i=2$), we have
\begin{equation}
\label{OptCondExpVal}
(1+\theta_1)E\left[e^{R(X_1-Z_1+X_2-Z_2)}\right] =
e^{RM_1}E\left[ e^{R(X_2-Z_2)}|X_1=x_1\right], 
\end{equation}
for $\mu_{X_1}$-almost every $x_1 \geq M_1$.
If $X_1$ and $X_2$ are dependent through a given copula $C(u,v)$, then
\begin{eqnarray*}
		E\left[e^{R(X_2-Z_2)}| X_1=x_1\right] = 
e^{RM_2}-\int_{0}^{M_2}\left(e^{RM_2}-e^{Rx_2}\right)
	\frac{\partial^2C}{\partial u\partial v}(F_1(x_1),F_2(x_2))f_2(x_2)\, dx_2, 
\end{eqnarray*}
which, from \eqref{OptCondExpVal}, must be constant for $x_1 \geq M_1$, implying that:
\begin{equation}
\label{ConstExpr}
\int_{0}^{M_2} \left(e^{RM_2}-e^{Rx_2}\right)\frac{\partial^3
	C}{\partial u^2\partial v}(u,F_2(x_2))f_2(x_2)\,\,dx_2 = 0, \quad
\forall u\in [F_1(M_1),1].
\end{equation}

Consider now that the risks are dependent through a copula: 
\begin{equation}
\label{Copula}
\displaystyle C(u,v)=uv\left( 1+(u-1)(v-1)\alpha
\right), \qquad \text{with } \alpha \in ]0,1]. 
\end{equation}
Since
\[
\frac{\partial^2 C}{\partial u^2} = 2v(v-1)\alpha \leq 0, 
\quad
\frac{\partial^2 C}{\partial v^2} = 2u(u-1)\alpha \leq 0,
\qquad
\forall (u,v) \in [0,1]^2,
\]
the risks $X_1$
and
$X_2$ are positively dependent in stochastic ordering.

We claim that, for copula \eqref{Copula}, the left-hand side of \eqref{ConstExpr} is strictly decreasing with respect to $M_2\in [0,+\infty[$ and therefore $M_2=0$ is the unique value satisfying condition \eqref{ConstExpr}.
To see this, notice that
$ \frac{\partial^3 C}{\partial u^2\partial v} = 2(2v-1)\alpha$ and, the derivative with respect to $M_2$ of the left-hand side of 
\eqref{ConstExpr} is
\begin{eqnarray*}
	\lefteqn{
		\frac{\partial}{\partial M_2} \int_{0}^{M_2} \left(e^{RM_2}-e^{Rx_2}\right)\frac{\partial^3
			C}{\partial u^2\partial v}(u,F_2(x_2))f_2(x_2)\,\,dx_2 = } &&  \\
	&\hspace{4cm}=& R e^{RM_2} 2\alpha F_2(M_2)\left[F_2(M_2)-1\right] <0.
\end{eqnarray*}
Hence, stop-loss is optimal only if it is optimal to cede the totality of the risk.

\subsection{Variance-related principles}
\label{SS variance-related}

If the reinsurance premium is computed by a variance-related premium calculation principle, then the optimal treaty for that risk is deterministic and the optimality conditions are as follows.

\begin{corollary}
\label{C variance}
Suppose that $U$ is differentiable and concave in $]-\infty,c[$, and let $\nu \in \overline{\mathcal H}_X$ be an optimal randomized strategy, and suppose that assumption {\bf (A2)} of Theorem \ref{T optimality condition relaxed} holds

If the marginal distribution $\mu_{X_i}$ is absolutely continuous except possibly for an atom at $x_i=0$, and the reinsurance premium for risk $i \in \{1,2,\ldots, n\}$ is computed by a variance-related principle 
\[
P_i(\nu) = \mathbb E^\nu Z_i + g \left( \mathbb E^\nu Z_i^2 - \left( \mathbb E^\nu Z_i \right)^2 \right) ,
\]
with $g$ continuous in $[0,+\infty[$, continuously differentiable in $]0,+\infty[$, monotonically increasing, then the optimal reinsurance for risk $i$ is a deterministic function $Z_i \in \mathcal Z_i$ such that
\begin{align*}
&
\mathbb E^\nu _{\left.\left( X_{[i]},Z_{[i]} \right) \right| X_i} U^\prime (L_\nu)(x_i,0) \leq 
\left( 1 - 2 \mathbb E Z_i g^\prime (\mathrm{Var}(Z_i))  \right) \mathbb E^\nu U^\prime(L_\nu) & \text{if } Z_i(x_i)=0,
\\ &
\mathbb E^\nu _{\left.\left( X_{[i]},Z_{[i]} \right) \right| X_i} U^\prime (L_\nu)(x_i,Z_i(x_i)) = 
\left( 1 + 2 \left(Z_i(x_i) - \mathbb E Z_i\right) g^\prime (\mathrm{Var}(Z_i))  \right) \mathbb E^\nu U^\prime(L_\nu) & \text{if } 0<Z_i(x_i)<x_i,
\\ &
\mathbb E^\nu _{\left.\left( X_{[i]},Z_{[i]} \right) \right| X_i} U^\prime (L_\nu)(x_i,x_i) \geq 
\left( 1 + 2 \left(x_i - \mathbb E Z_i\right) g^\prime (\mathrm{Var}(Z_i))  \right) \mathbb E^\nu U^\prime(L_\nu) & \text{if } Z_i(x_i)=x_i
\end{align*}
holds for $\mu_{X_i}$-almost every $x_i$.
\end{corollary}

\begin{proof}
Under the assumptions of the Corollary,
\begin{align*}
\Psi_i(u_1,u_2)= u_1 + g(u_2-u_1^2)
\end{align*}
Is continuous in $D_2 \cup \left\{ (u_1,u_2) \in [0,+\infty[^2: u_1 = u_2^{\frac 12} \right\}$, continuously differentiable in $D_2$, and
\begin{align*}
\frac{\partial \Psi_i}{\partial u_1} (u_1,u_2) = 1 - 2 u_1 g^\prime \left(u_2 -u_1^2 \right) ,
\qquad
\frac{\partial \Psi_i}{\partial u_2} (u_1,u_2) = 
g^\prime \left(u_2 -u_1^2 \right) .
\end{align*}
By assumption, $\frac{\partial \Psi_i}{\partial u_2} \geq 0$, and clearly $\frac{\partial \Psi_i}{\partial u_1} $, $\frac{\partial \Psi_i}{\partial u_2} $ have no common zeros.
Thus, Corollary \ref{C higher moments} guarantees that the optimal strategy for risk $i$  is deterministic, and the optimality conditions follow from Theorem \ref{T optimality condition deterministic}.
\end{proof}

Notice that the assumptions of Corollary \ref{C variance} (and a fortiori, Theorem \ref{T deterministic treaties}) include cases where the premium calculation principle is not a convex functional in the space of deterministic strategies.
For details, see the characterization of convex variance related principia in \cite[Proposition 1]{GC10b}.

\section{Computation of optimal strategies}\label{S algorithm}

There is little hope of finding closed form solutions for the optimality conditions provided in Theorems \ref{T optimality condition relaxed} or \ref{T optimality condition deterministic}.
Even in cases where expected utility and moments can be explicitely computed for many common treaties (say, risks exponentially distributed and exponential utility function) and the assumptions of Corollary \ref{C higher moments} hold (say, premia are computed by expected value principles), the optimality conditions in Theorem \ref{T optimality condition deterministic} translate into transcendental equations which have no closed form solutions.
Thus, any practical application of the results in the previous sections must rely on numerical solutions.

A full analysis of a numerical scheme is outside the scope of the present text, and certainly there are many details that depend on the particular distributions, utility function, and premia calculation principles.
Below, we present a sketch of a general analysis, focusing on the solutions of the optimality conditions provided in Theorem \ref{T optimality condition deterministic}, under the assumptions of Corollary \ref{C higher moments}.

In Section \ref{SS fixed points}, we show that the optimality conditions are equivalent to a certain fixed point problem.
From this, we outline a numerical scheme in Section \ref{SS numerical scheme}.
In Section \ref{SS numerical issues}, we discuss the main issues arising in a practical implementation of the numerical scheme and argue that the scheme is feasible, even when the number of risks is moderately large.
Section \ref{SS beyond deterministic} contains a brief discussion on how the scheme can be generalized to the general optimality conditions provided in Theorem \ref{T optimality condition relaxed}.
In Section \ref{SS numerical example} we present numerical examples illustrating the application of the proposed numerical approach.

\subsection{Optimality conditions as a fixed point problem}\label{SS fixed points}

In this section, as well as in Sections \ref{SS numerical scheme} and \ref{SS numerical issues}, we will always assume that the number of risks is $n \geq 2$ and the assumptions of Corollary \ref{C higher moments} hold for $i=1,2, \ldots, n$.
\vspace{5mm}

Fix a reinsurance strategy $\hat Z=(\hat Z_1, \hat Z_2, \ldots, \hat Z_n) \in \mathcal Z$.
For each
$m=\left(m_0, m_1,\ldots , m_n \right) \in \mathbb R \times \overline D_{k_1} \times \cdots \times \overline D_{k_n}$, consider the functions $\Lambda_i^{m_{[0,i]},\hat Z_{[i]}} $, defined in the domain $\left\{ (x,z): x \geq 0, -P_i(X_i) \leq z \leq x \right\}$ by
\begin{align*}
\Lambda_i^{m_{[0,i]},\hat Z_{[i]}} (x,z) =
\left. \mathbb E\right|^{\mu_X}_{X_i=x} U^\prime \left( c - \sum_{j \neq i} \left( X_j + \Psi_j(m_j) - \hat Z_j(X_j)  \right) - X_i + z \right),
\end{align*}
with $P_i(X_i)$ being the premium charged for cedence of the full risk $i$, and $m_{[0,i]}$ being the array $m$ with the elements $m_0$ and $m_i$ omitted.
Under the assumptions of Corollary \ref{C higher moments}, the function
\[
G_i^{m_{[0,i]},\hat Z_{[i]}}(m_0,m_i,x_i,z_i) = \Lambda_i^{m_{[0,i]}, \hat Z_{[i]}} (x_i , z_i - \Psi_i(m_i) ) - m_0 \sum_{j=1}^{k_i} \frac{\partial \Psi_i}{\partial u_j}(m_i) j z_i^{j-1}
\]
is strictly decreasing with respect to $z_i$ in the interval $[0,x_i]$.
Thus, the equation
\[
G_i^{m_{[0,i]},\hat Z_{[i]}}(m_0,m_i,x_i,z_i) = 0
\]
defines implicitely one unique function $z_i=Z_i^{m_{[0,i]},\hat Z_{[i]}}(m_0,m_i,x_i)$ such that
\begin{align}
& \label{Eq Z implicit 1}
Z_i^{m_{[0,i]},\hat Z_{[i]}}(m_0,m_i,x_i) = 0,
\qquad \text{if }
G_i^{m_{[0,i]},\hat Z_{[i]}}(m_0,m_i,x_i,0) \leq 0,
\\ & \label{Eq Z implicit 2}
Z_i^{m_{[0,i]},\hat Z_{[i]}}(m_0,m_i,x_i) = x_i,
\qquad \text{if }
G_i^{m_{[0,i]},\hat Z_{[i]}}(m_0,m_i,x_i,x_i) \geq 0,
\\ & \label{Eq Z implicit 3}
G_i^{m_{[0,i]},\hat Z_{[i]}}(m_0,m_i,x_i,Z_i^{m_{[0,i]},\hat Z_{[i]}}(m_0,m_i,x_i)) = 0,
\qquad \text{in other cases}.
\end{align}
That is, given the treaties $\hat Z_{[i]}$ and the parameters $m_{[0,i]}$, the scheme above defines uniquely a family of treaties for risk $i$, depending on the parameters $m_0$ and $m_i$.

We introduce the functions $\Upsilon_i^{m_{[0,i]},\hat Z_{[i]}}: \mathbb R \times \overline D_{k_i} \mapsto  \mathbb R \times \overline D_{k_i}$, defined as
\begin{align*}
&
\Upsilon_i^{m_{[0,i]},\hat Z_{[i]}}(m_0,m_i) =
\\ = &
\mathbb E^{\mu_X} \Bigg(
U^\prime \left( 
L_{Z_i^{m_{[0,i]},\hat Z_{[i]}}(m_0,m_i,X_i), \hat Z_{[i]}}
\right),
Z_i^{m_{[0,i]},\hat Z_{[i]}}(m_0,m_i,X_i),
\left(Z_i^{m_{[0,i]},\hat Z_{[i]}}(m_0,m_i,X_i)\right)^2,
\ldots 
\\ & \hspace{70pt}
\ldots ,
\left(Z_i^{m_{[0,i]},\hat Z_{[i]}}(m_0,m_i,X_i)\right)^{k_i}
\Bigg) 
\\ = &
\mathbb E^{\mu_X} \Bigg(
\Lambda_i^{m_{[0,i]}, \hat Z_{[i]}}
\left( X_i,
Z_i^{m_{[0,i]},\hat Z_{[i]}}(m_0,m_i,X_i) - \Psi_i(m_i)
\right),
Z_i^{m_{[0,i]},\hat Z_{[i]}}(m_0,m_i,X_i),
\\ & \hspace{70pt}
\left(Z_i^{m_{[0,i]},\hat Z_{[i]}}(m_0,m_i,X_i)\right)^2,
\ldots ,
\left(Z_i^{m_{[0,i]},\hat Z_{[i]}}(m_0,m_i,X_i)\right)^{k_i}
\Bigg) .
\end{align*}
The reinsurance treaty $\hat Z_i$ satisfies the optimality condition in Theorem \ref{T optimality condition deterministic} if and only if $(m_0,m_i)$ is a fixed point of $\Upsilon_i^{m_{[i]}, \hat Z_{[i]}}$, and the parameters $m_{[0,i]}$ are
\begin{align*}
m_j = \mathbb E^{\mu_X} \left(
\hat Z_j(X_j), \hat Z_j(X_j)^2, \ldots , \hat Z_j(X_j)^{k_j}
\right)
\qquad \text{for } j \neq i.
\end{align*}
Under the hypothesis of Corollary \ref{C higher moments}, this implies that $\hat Z_i$ is optimal for risk $i$ given the policies $\hat Z_{[i]}$ for the other risks.
The strategy $\hat Z = \left( \hat Z_1, \hat Z_2, \ldots, \hat Z_n \right)$ satisfies the optimality conditions of Theorem \ref{T optimality condition deterministic} if and only if for every $i \in \{1,2, \ldots, n\}$, $(m_0,m_i)$ is a fixed point of the corresponding map $\Upsilon_i^{m_{[0,i]}, \hat Z_{[i]}}$.

\subsection{Outline of numerical scheme}\label{SS numerical scheme}

Based in the considerations above, we propose the following scheme to solve the optimality conditions given in Theorem \ref{T optimality condition deterministic}:

\begin{description}
\item[1.] (initialization):
\begin{description}
\item[1.1.] 
Pick initial treaties $Z^{(0)} = \left( Z^{(0)}_1, Z^{(0)}_2, \ldots, Z^{(0)}_n \right) \in \mathcal Z$.
\item[1.2.] 
Compute the initial vector $m^{(0)} = \left( m^{(0)}_0, m^{(0)}_1, \ldots , m^{(0)}_n \right)$:
\begin{align*}
&
m^{(0)}_0 : = 
\mathbb E^{\mu_X} 
U^\prime \left( 
L_{Z^{(0)}_1,Z^{(0)}_2,\ldots Z^{(0)}_n}
\right)
\\ &
m^{(0)}_j : =
\mathbb E^{\mu_X} \left(
Z^{(0)}_j(X_j), \left( Z^{(0)}_j(X_j \right)^2, \ldots , \left( Z^{(0)}_j(X_j \right)^{k_j}
\right), \qquad j = 1,2, \ldots , n
\end{align*}
\end{description} 
\item[2.]
(main loop)
Repeat until convergence:
\begin{description}
\item[2.1.]
Set
$ m^{(i+1)} := m^{(i)}$,
$ Z^{(i+1)} := Z^{(i)}$
\item[2.2.] 
For $j=1,2, \ldots , n$:
\begin{description}
\item[2.2.1.]
Compute the function $\Upsilon_j^{ m^{(i+1)}_{[0,j]}, Z^{(i+1)}_{[j]}}$
\item[2.2.2.]
Find $\left(\hat m_0, \hat m_j \right)$, a fixed point of $\Upsilon_j^{ m^{(i+1)}_{[0,j]}, Z^{(i+1)}_{[j]}}$, and the corresponding treaty $\hat Z(x)= Z_j^{m^{(i+1)}_{[0,j]}, Z^{(i+1)}_{[j]}}(\hat m_0, \hat m_j,x)$
\item[2.2.3.]
Update
$m^{(i+1)}_0 := \hat m_0$,
$\,\,\, m^{(i+1)}_j := \hat m_j$,
$\,\,\, Z^{(i+1)} := \hat Z$
\end{description}
\end{description}
\end{description}

\subsection{Issues in practical implementation of the numerical scheme} \label{SS numerical issues}

The scheme above is outlined in the broadest possible terms, and it must be translated into some concrete algorithm to be applied to any particular model.
Now, we discuss the main difficulties arising in the construction of viable algorithms and show how they can be overcome.

For brevity, we drop the superscripts, writing $Z_j$, $\Lambda_j$, $G_j$, etc., for $Z_j^{m_{[0,j]},Z_{[j]}}$, $\Lambda_j^{m_{[0,j]},Z_{[j]}}$, $G_j^{m_{[0,j]},Z_{[j]}}$, etc., being understood that such functions are defined based on some given parameters $m_{[0,j]}$ and some given functions $Z_{[j]}$.

\subsubsection{Computation of fixed points}

Due to its simplicity and quick convergence, Newton's algorithm seems an attractive method to compute the fixed points of the functions $\Upsilon_j^{ m^{(i+1)}_{[0,j]}, Z^{(i+1)}_{[j]}}$.
However, due to the constraints $0 \leq Z_j(x) \leq x$, the map $(m_0,m_j) \mapsto \Upsilon_j^{ m^{(i+1)}_{[0,j]}, Z^{(i+1)}_{[j]}}(m_0,m_j)$ may fail to be differentiable, even when the utility function $U$ and the premium calculation principle $\Psi_j$ are smooth.
This difficulty can be avoided by the following scheme.

Let $A= \{(x,z): 0 < z < x \}$, and consider a family of smooth functions $\left\{ \beta_\varepsilon: A \mapsto \mathbb R \right\}_{\varepsilon>0}$, such that:
\begin{align*}
&
\frac{\partial \beta_\varepsilon}{\partial z}(x,z) \leq 0, 
\qquad
\text{for every } 0 < z <x,
\\ &
\lim_{z \to 0^+} \beta_\varepsilon(x,z) = + \infty
\quad \text{and} \quad
\lim_{z \to x^-} \beta_\varepsilon(x,z) = - \infty,
\qquad 
\text{for every } x>0,
\\ &
\lim_{\varepsilon \to 0^+} \beta_\varepsilon(x,z) = 0,
\qquad
\text{uniformly in any compact set } K \subset A. 
\end{align*}
A suitable family is, for example, $\beta_\varepsilon (x,z)= \varepsilon \frac{x^{\alpha +1}}{1+x^{\alpha +1}} \left( \frac 1{z^\alpha} - \frac 1{(x-z)^\alpha} \right)$, with $\alpha>0$ constant.
Then, the equation
\begin{align*}
G_j (m_0,m_j,x,z) + \beta_\varepsilon (x,z)= 0
\end{align*}
defines implicitely one unique function $
Z_{j,\varepsilon} (m_0,m_j,x)$.
This function satisfies $0 < Z_{j,\varepsilon} (m_0,m_j,x) < x$ for every $x>0$, and $\lim\limits_{\varepsilon \to 0^+} Z_{j,\varepsilon} (m_0,m_j,x) = Z_j (m_0,m_j,x)$ uniformly with respect to $x$ on bounded intervals.
Substituting the approximation $Z_{j,\varepsilon}$ for $Z_j$ in the definition of the function $\Upsilon_j$, we obtain the approximation
\begin{align*}
&
\Upsilon_{j,\varepsilon} (m_0,m_j) =
\\ = &
\mathbb E^{\mu_X} \Bigg(
\Lambda_j
\left( X_j,
Z_{j,\varepsilon} (m_0,m_j,X_j) - \Psi_j(m_j)
\right),
Z_{j,\varepsilon} (m_0,m_j,X_j),
\left(Z_{j,\varepsilon} (m_0,m_j,X_j) \right)^2, \ldots
\\ & \hspace{70pt}
\ldots ,
\left(Z_{j,\varepsilon}(m_0,m_j,X_j)\right)^{k_j}
\Bigg) .
\end{align*}
Since the constraint $0 \leq Z_j(x) \leq x$ is not active when $Z_{j,\varepsilon}$ is substituted for $Z_j$, the function $\Upsilon_{j,\varepsilon} $ is differentiable, provided $U$ and $\Psi_j$ are sufficiently regular.

To compute the Jacobian matrix of $\Upsilon_{j,\varepsilon}$, consider the function 
\begin{align*}
d\Lambda_j (x,z) =
\left. \mathbb E\right|^{\mu_X}_{X_j=x} U^{\prime \prime} \left( c - \sum_{\ell \neq j} \left( X_\ell + \Psi_\ell(m_\ell) - Z_\ell(X_\ell)  \right) - X_j + z \right)
\end{align*}
(there is some abuse of notation in the symbol for this function).
Then, writing $m_j = \left(m_{j,1},m_{j,2}, \ldots , m_{j,k_j} \right)$, and $Z_{j,\varepsilon}$ for $Z_{j,\varepsilon}(m_0,m_j,X_j)$:
\begin{align*}
&
\frac{\partial}{\partial m_0} \mathbb E^{\mu_X} 
\Lambda_j
\left( X_j,
Z_{j,\varepsilon} 
\right) =
\mathbb E^{\mu_X} \left(
d \Lambda_j
\left( X_j,
Z_{j,\varepsilon} - \Psi_j
\right) 
\frac{\partial Z_{j,\varepsilon}}{\partial m_0} 
\right),
\\ &
\frac{\partial}{\partial m_{j,\ell}} \mathbb E^{\mu_X} 
\Lambda_j
\left( X_j,
Z_{j,\varepsilon} 
\right) =
\mathbb E^{\mu_X} \left(
d \Lambda_j
\left( X_j,
Z_{j,\varepsilon} - \Psi_j
\right) 
\left(\frac{\partial Z_{j,\varepsilon}}{\partial m_{j,\ell}} - \frac{\partial \Psi_j}{\partial m_{j,\ell}} \right)
\right),
\\ &
\frac{\partial}{\partial m_0} \mathbb E^{\mu_X} 
\left( Z_{j,\varepsilon} 
\right)^r =
r E^{\mu_X} 
\left(\left( Z_{j,\varepsilon} 
\right)^{r-1} \frac{\partial Z_{j,\varepsilon}}{\partial m_0} \right)
\qquad
r=1,2,\ldots, k_j,
\\ &
\frac{\partial}{\partial m_{j,\ell}} \mathbb E^{\mu_X} 
\left( Z_{j,\varepsilon} 
\right)^r =
r E^{\mu_X} 
\left(\left( Z_{j,\varepsilon} 
\right)^{r-1} \frac{\partial Z_{j,\varepsilon}}{\partial m_{j,\ell}} \right)
\qquad
r=1,2,\ldots, k_j,
\end{align*}
with
\begin{align*}
&
\frac{\partial Z_{j,\varepsilon}}{\partial m_0} = 
\frac{\sum\limits_{r=1}^{k_j} \frac{\partial \Psi_j}{\partial m_{j,r}}r Z_{j,\varepsilon}^{r-1}}{  d\Lambda_j(x,Z_{j,\varepsilon}-\Psi_j) - \sum\limits_{r=2}^{k_j} \frac{\partial \Psi_j}{\partial m_{j,r}}r(r-1) Z_{j,\varepsilon}^{r-2} +
	\frac{\partial \beta}{\partial z}(x,Z_{j, \varepsilon}) },
\\ &
\frac{\partial Z_{j,\varepsilon}}{\partial m_{j,\ell}} = 
\frac{
d\Lambda_j(x,Z_{j,\varepsilon}-\Psi_j) \frac{\partial \Psi_j}{\partial m_{j,\ell}} + m_0 \sum\limits_{r=1}^{k_j} \frac{\partial^2 \Psi_j}{\partial m_{j,\ell} \partial m_{j,r}} r Z_{j,\varepsilon}^{r-1}
}{  d\Lambda_j(x,Z_{j,\varepsilon}-\Psi_j) - \sum\limits_{r=2}^{k_j} \frac{\partial \Psi_j}{\partial m_{j,r}}r(r-1) Z_{j,\varepsilon}^{r-2} +
\frac{\partial \beta}{\partial z}(x,Z_{j, \varepsilon}) }.
\end{align*}

\subsubsection{Convergence criteria}

A natural convergence criterion is convergence in the variables $m$.
That is, to exit the main loop when $\| m^{i+1} - m^i \|$ is less than some tolerance.
However, convergence in this sense may fail to occur.
Since each iteration of step 2.2.2 selects the treaty for risk $j$ given the current choice of the remaining treaties, the existence of limit cycles is not easy to rule out.
However, under the assumptions of Corollary \ref{C higher moments}, step 2.2.2 computes a treaty $Z_j$ which is optimal for risk $j$ given the current choices of the remaining treaties $Z_{[j]}$ (for $\ell < j$, $Z_\ell$ is a treaty already computed during the current cycle of the main loop, while for $\ell >j$ $Z_\ell$ is a treaty computed in the previous cycle).
Thus, the overall performance of the strategy $Z=(Z_1,Z_2, \ldots, Z_n)$ improves at each iteration of step 2.2.2 and, a fortiori, at each iteration of the main loop.
Thus, to exit the main loop when $\mathbb E U(L_{Z^{(i+1)}}) - \mathbb E U(L_{Z^{(i)}})$ is less than some tolerance is a criterion that guarantees convergence in a finite number of cycles and provides a strategy where the treaty for each risk is (approximately) optimal, given the choice of the remaining treaties.

\subsubsection{Large number of risks}

The general scheme outlined in Section \ref{SS numerical scheme} translates into algorithms which are nontrivial but are well within the possibilities of common desktop computers, when the number of risks is very small (say, two or three).
However, for practical applications, an algorithm able to deal at least with a moderately large number of risks is highly desirable.

Each evaluation of the functions $\Lambda_j$ and $d\Lambda_j$ is a quadrature in $\mathbb R_+^{n-1}$, and the number of such evaluations is proportional to $n$, the number of risks.
Thus, if quadratures are computed by usual rectangular mesh methods (say, Gauss, etc.), the computational workload grows exponentially with respect to the number of risks and therefore the algorithm cannot be applied beyond very small numbers of risks.
However, there are quadrature methods that do not rely on rectangular meshes.
Such methods are typically quite inefficient in low dimensions but their workload grows at lower rates with the number of dimensions.

Examples of such methods which are easy to implement are Monte Carlo methods.
Their efficiency depends mainly on the sample size required to achieve the desired accuracy.
We argue briefly that, for the problem under consideration, this sample size is likely to grow very slowly with the number of risks.

Let $X= \sum\limits_{i=1}^n X_i$ be the aggregate risk.
It is natural to assume that $X$ represents the total (actuarial) risk exposure of a given insurance company, and that different values of $n$ reflect different levels of desaggregation of the total risk and hence different levels of detail in its management.
Under this scenario, the aggregate risk does not depend on $n$ and
\begin{align}
\label{Eq desagregated bound}
0 \leq
\sum_{i=1}^n\left(X_i + \Psi_i(m_i) - Z_i(X_i) \right) \leq
X + \sum_{i=1}^n P_i(X_i)
\end{align}
holds for every $n$.

The sample size required to obtain a given accuracy depends mainly in the tail behaviour of the quantity being computed, which in our setting is a function of the middle term in \eqref{Eq desagregated bound}.
Thus, inequality \eqref{Eq desagregated bound} suggests that any increase in tail weight when $n$ increases is driven by subadditivities in the premium calculation principles, and therefore should be quite modest.

Under the assumptions above, a version of our scheme where one-dimensional quadratures are computed by mesh methods and $(n-1)$-dimensional quadratures are computed by Monte Carlo generates a computational loading that grows only slightly supralinearly with the dimension $n$.
This is roughly the lowest growth that can realistically be expected in a problem of increasing number of dimensions.

\subsection{Beyond the assumptions of Corollary \ref{C higher moments}} \label{SS beyond deterministic}

If the assumptions of Corollary \ref{C higher moments} fail, then we expect optimal strategies to be randomized strategies.

In that case, the (random) treaty for risk $i$ can be described as a family of Borel probability measures $\left\{ \nu_{i,x} \right\}_{x \geq 0}$ such that, for each $x\geq 0$, $\nu_{i,x}$ is concentrated in the interval $[0,x]$. 
We can follow the same argument as in Section \ref{SS fixed points}, provided we redifine the function $\Lambda_i$ as
\begin{align*}
\Lambda_i
(x,z)= \mathbb E^\nu_{(X_{[i]},Z_{[i]})|X_i=x}U^\prime \left(c - \sum_{j \neq i}(X_j + \Psi_j(m_j)-Z_j) - X_i + z \right).
\end{align*}
Conditions \eqref{Eq Z implicit 1}, \eqref{Eq Z implicit 2} and \eqref{Eq Z implicit 3} are generalized to:
\begin{align}
& \label{Eq nu implicit 1}
\nu_{i,x} = \delta_0,
\qquad \text{if }
\left\lceil \Lambda_i \right\rceil (x,z) - m_0 \sum_{j=1}^{k_i} \frac{\partial \Psi_i}{\partial u_j}(m_i) j z^{j-1} < 0 \quad \forall z \in [0,x],
\\ & \label{Eq nu implicit 2}
\nu_{i,x} = \delta_x,
\qquad \text{if }
\left\lfloor \Lambda_i \right\rfloor (x,z) - m_0 \sum_{j=1}^{k_i} \frac{\partial \Psi_i}{\partial u_j}(m_i) j z^{j-1} > 0 \quad \forall z \in [0,x],
\\ & \notag
\text{the support of } \nu_{i,x} \text{  is contained in the set} 
\\ & \hspace{10mm} \notag
A_{i,x} =
\left\{
z \in [0,x]: 
\left\lfloor \Lambda_i \right\rfloor (x,z) \leq m_0 \sum_{j=1}^{k_i} \frac{\partial \Psi_i}{\partial u_j}(m_i) j z^{j-1} \leq \left\lceil \Lambda_i \right\rceil (x,z)
\right\},
\\ & \label{Eq nu implicit 3}
\hspace{110mm} \text{in other cases},
\end{align}
where $\delta_a$ denotes the Dirac measure concentrated at the point $a$.

Since condition \eqref{Eq nu implicit 3} does not define one unique measure, some further condition is required to obtain a generalization of function $\Upsilon_i$.

One possible way to overcome this difficulty is to chose $\nu_{i,x}$ minimizing the distance between the points $(m_0,m_i)$ and $\mathbb E^{\nu_{i,x}} \left( \Lambda_i\left(x, Z_i - \Psi_i(m_i)\right) , Z_i,Z_i^2, \ldots , Z_i^{k_i} \right)$.
In general, this does not yield a unique measure $\nu_{i,x}$ but, since the sets
\begin{align*}
\left\{
\mathbb E^\eta \left( \Lambda_i\left(x, Z_i - \Psi_i(m_i)\right) , Z_i,Z_i^2, \ldots , Z_i^{k_i} \right) : \mathrm{supp}(\eta) \subset A_{i,x}
\right\}
\qquad
x \geq 0
\end{align*}
are convex and closed, it yields one unique Borel function $x \mapsto \mathbb E^{\nu_{i,x}} \left( \Lambda_i\left(x, Z_i - \Psi_i(m_i)\right) , Z_i,Z_i^2, \ldots , Z_i^{k_i} \right)$, and the expression
\begin{align*}
&
\Upsilon_i (m_0,m_i) =
\int_0^{+\infty}
\mathbb E^{\nu_{i,x}} \left( \Lambda_i\left(x, Z_i - \Psi_i(m_i)\right) , Z_i,Z_i^2, \ldots , Z_i^{k_i} \right) \, \mu_{X_i}(dx)
\end{align*}
is a well defined function.

If the function $\Upsilon_i$ defined above is continuous, then Brower's fixed point theorem guarantees existence of a fixed point.

\subsection{Numerical illustration}
\label{SS numerical example}

We present here numerical examples illustrating the theoretical
results of this work and the proposed numerical algorithm. 
We consider two risks, $X_1$ and $X_2$, with distribution
functions given by $F_1(x)=1-e^{-x}$  and
$F_2(x)=1-\left(\frac{4}{4+x}\right)^5$, respectively, such that
$E[X_1]=E[X_2]=1$ and $Var[X_1]=1$ and $Var[X_2]=5/3$. 
Three different dependence structures are analysed, by means of
copulas, and two different premium calculation
principles are considered: (i) the expected value principle, in which case the loadings
are chosen to be $\theta_1=0.3$ and $\theta_2=0.5$ for risks $X_1$ and
$X_2$, respectively; (ii) the standard deviation principle, for
which the premium loadings are $\theta_1=\theta_2=0.5$. The results are
presented in comparison with the independence case.

Regarding the dependence structure, we will consider Frank's Copula, 
given by:
\begin{equation*}
\label{Frank}
C_{\alpha}(u_1,u_2)=-\frac{1}{\alpha}\log\left(
1+\frac{(e^{-\alpha\,
		u_1}-1)(e^{-\alpha\,
		u_2}-1)}{e^{-\alpha}-1}
\right) .
\end{equation*}
This copula is known to have no upper nor lower tail dependence. When
the copula parameter $\alpha=0$, the random variables are independent,
if $\alpha >0$ there is a positive dependence and when $\alpha<0$ the
dependence is negative. 
We consider two cases: $\alpha=10$, and
$\alpha=-10$. 
We also consider a copula, with no general expression, which includes
positive and negative dependencies in different regions of the domain.
This copula is defined by means of a $6\times 6$ matrix of values in the square $[0,1]^2$, using interpolation by the independence copula between those values. The matrix values are chosen to guarantee that the copula is grounded and 2d-increasing.
The numerical solutions of the optimal treaties are presented in Figure \ref{results}.

\begin{figure}[ht]
	\begin{center}
		\includegraphics[width=0.31\textwidth]{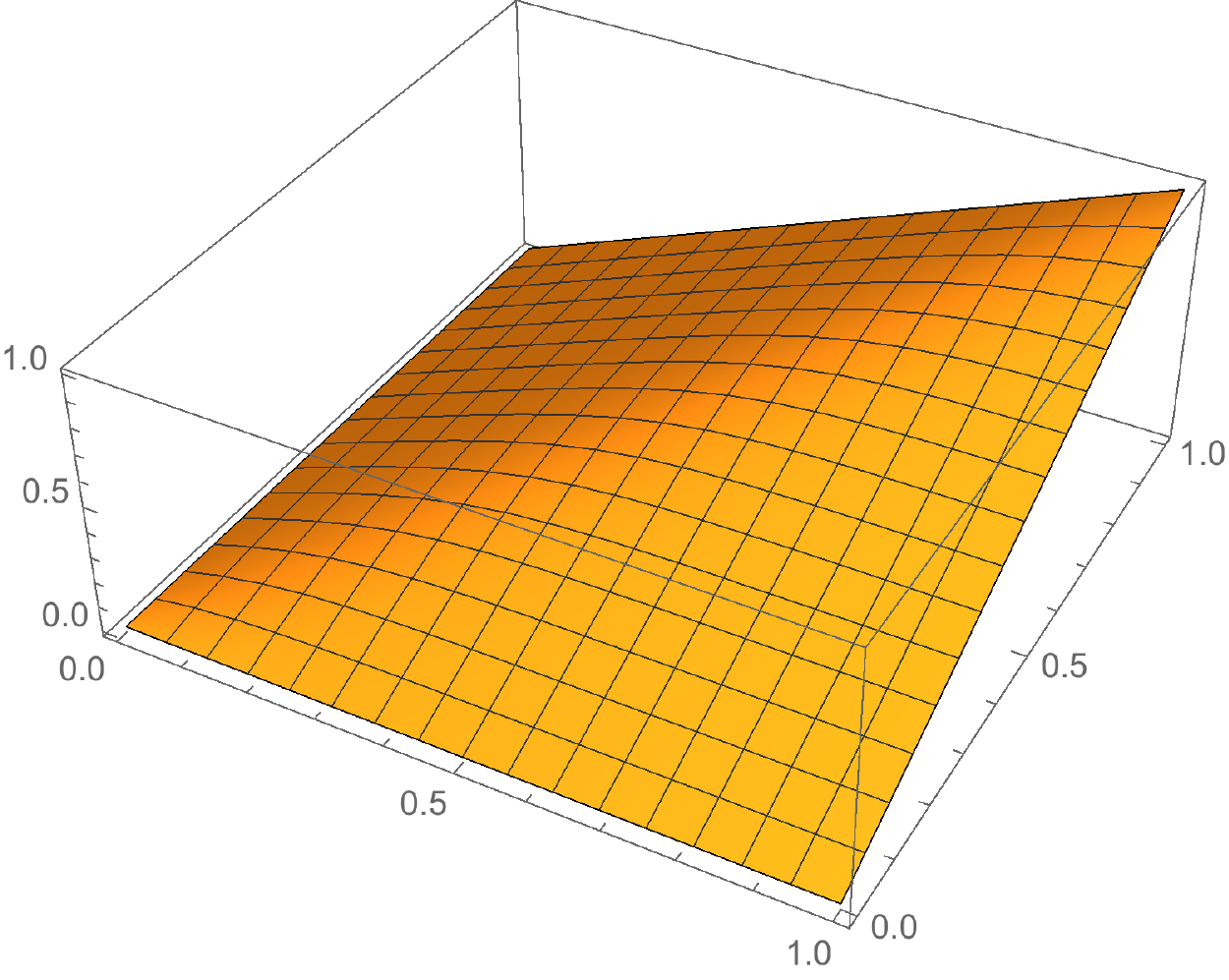} 
		\hspace{0.1cm}
		\includegraphics[width=0.31\textwidth]{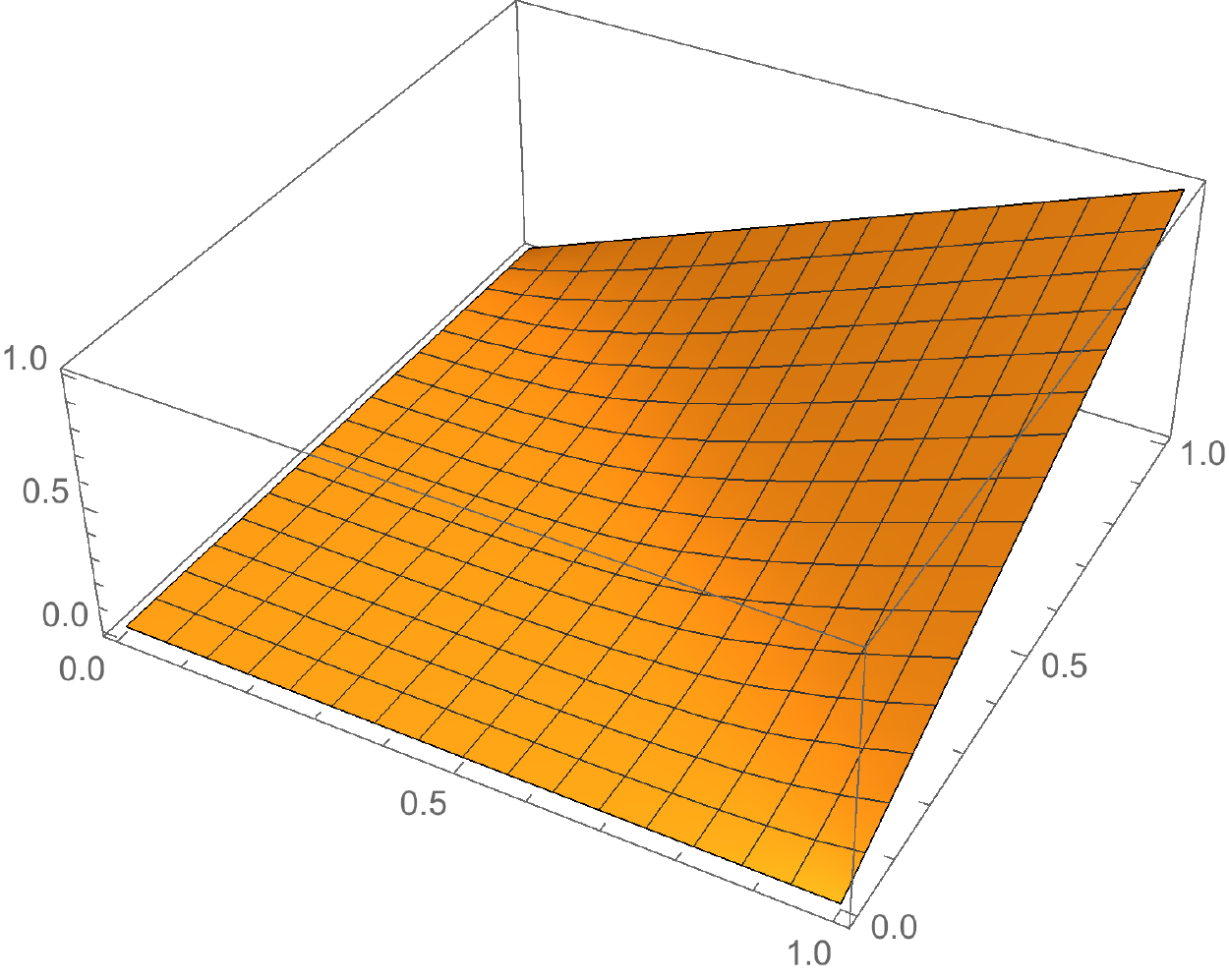} 
		\hspace{0.1cm}
		\includegraphics[width=0.31\textwidth]{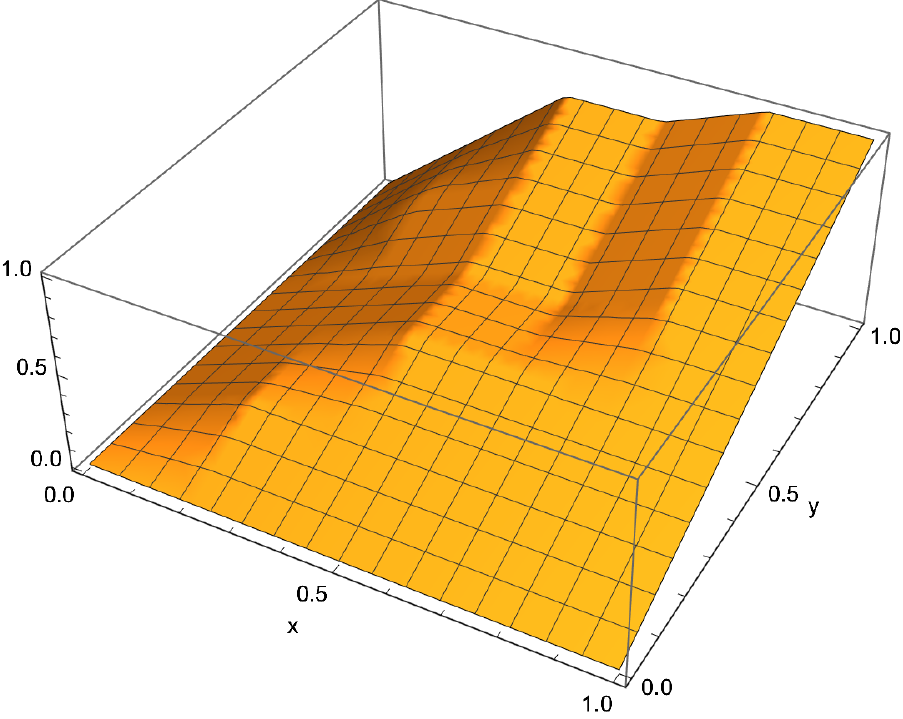} 
	\end{center}
	\caption{The three copulas considered. Left: Frank's copula
		with $\alpha=10$; Middle: Frank's copula with $\alpha=-10$; Right:
		copula including positive and negative dependencies.}
	\label{Copulas}
\end{figure}

\begin{figure}[ht]
	\begin{center}
		\includegraphics[width=0.24\textwidth]{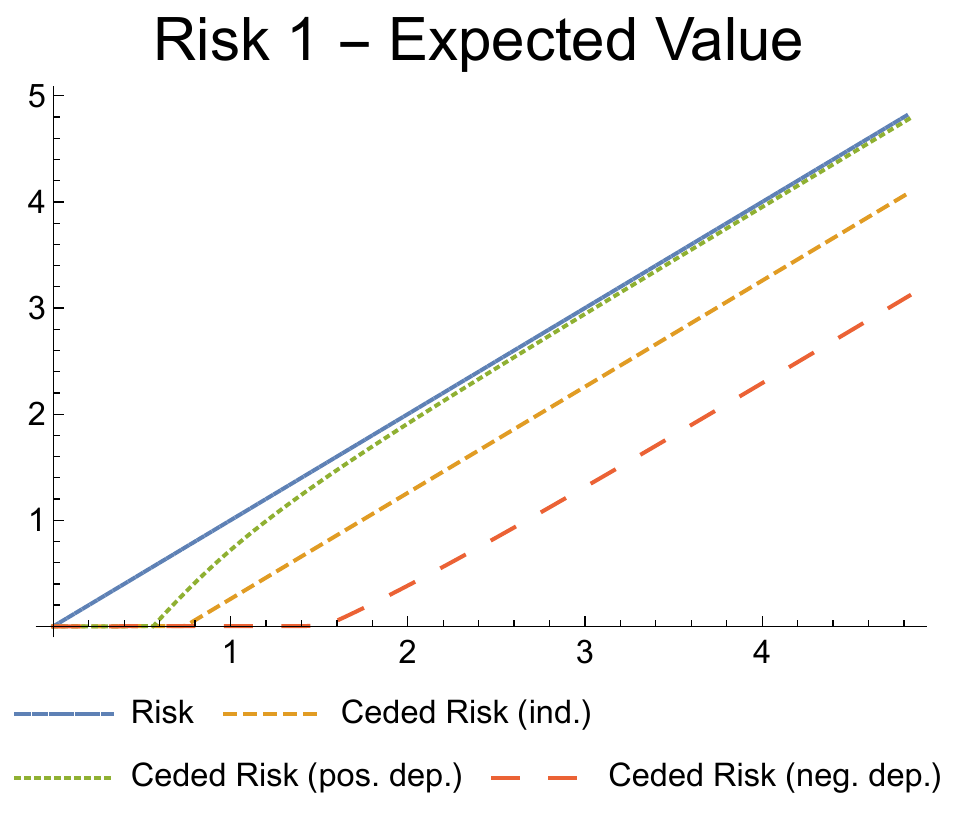}
		\hspace{0.0cm}
		\includegraphics[width=0.24\textwidth]{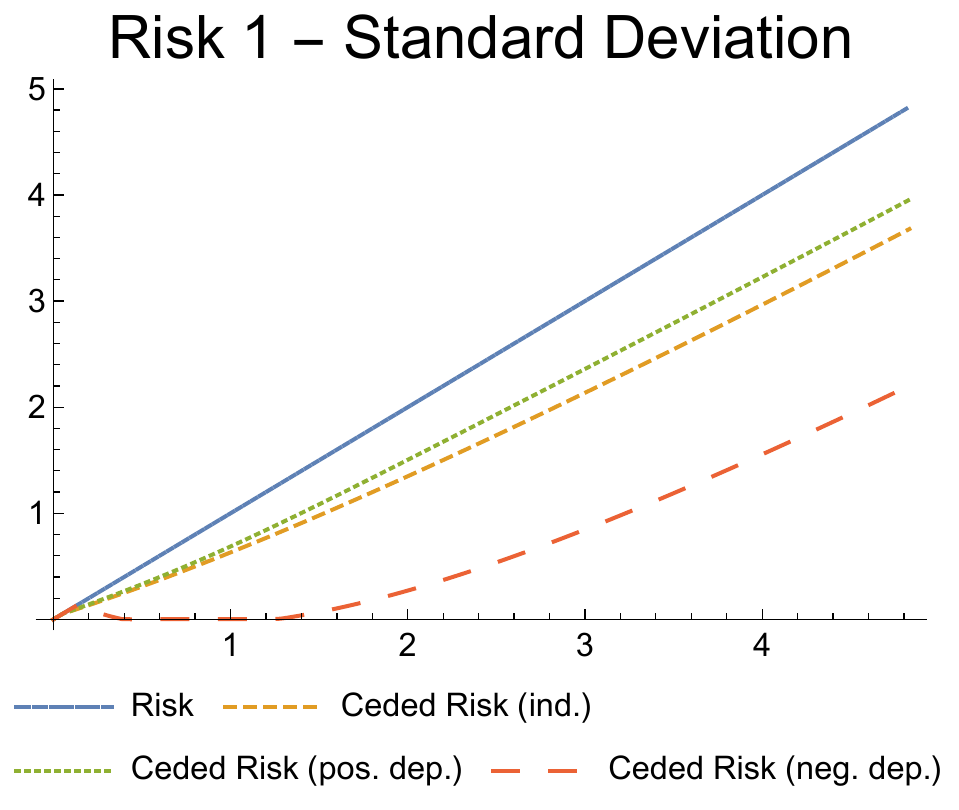}
		\hspace{0.0cm}
		\includegraphics[width=0.22\textwidth]{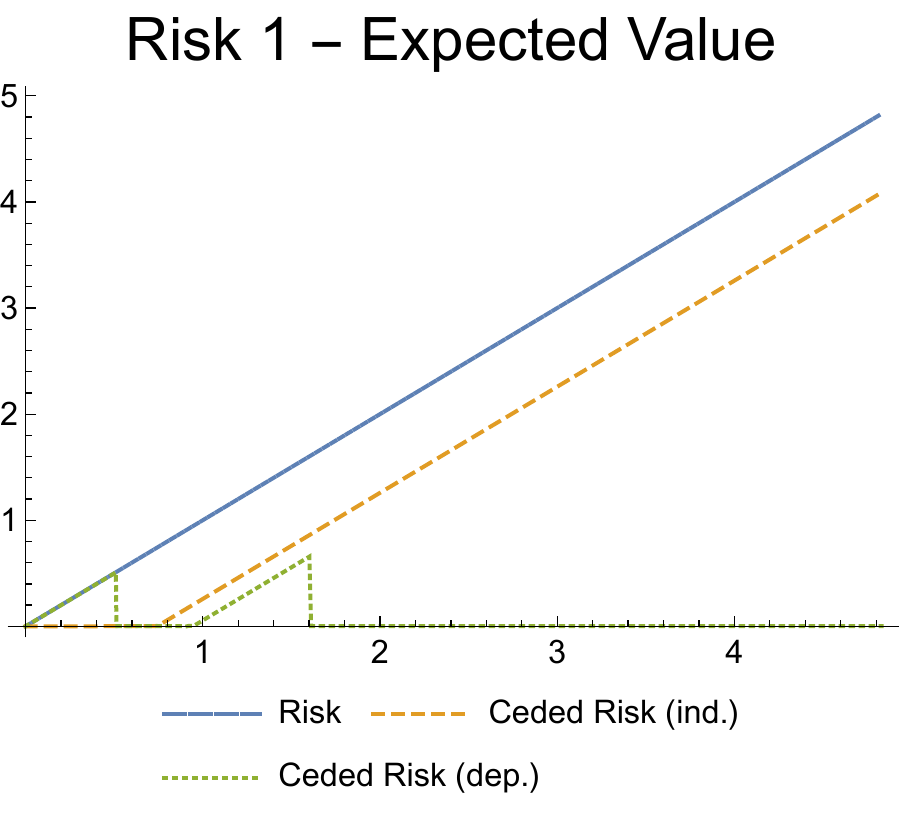}
		\hspace{0.0cm}
		\includegraphics[width=0.22\textwidth]{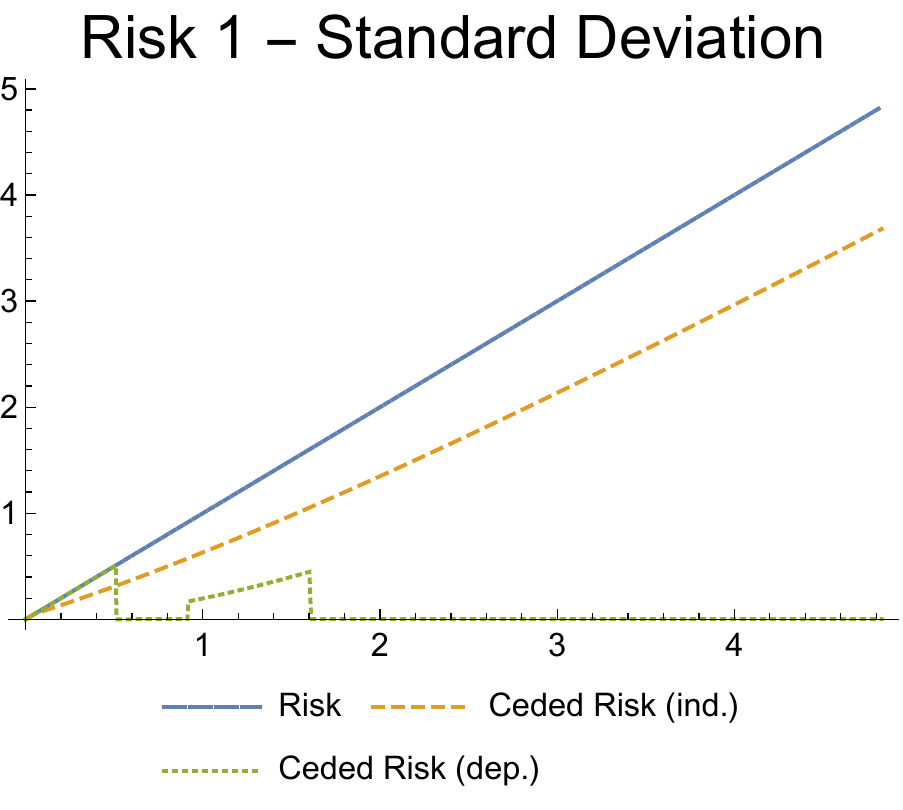}
		
		\includegraphics[width=0.24\textwidth]{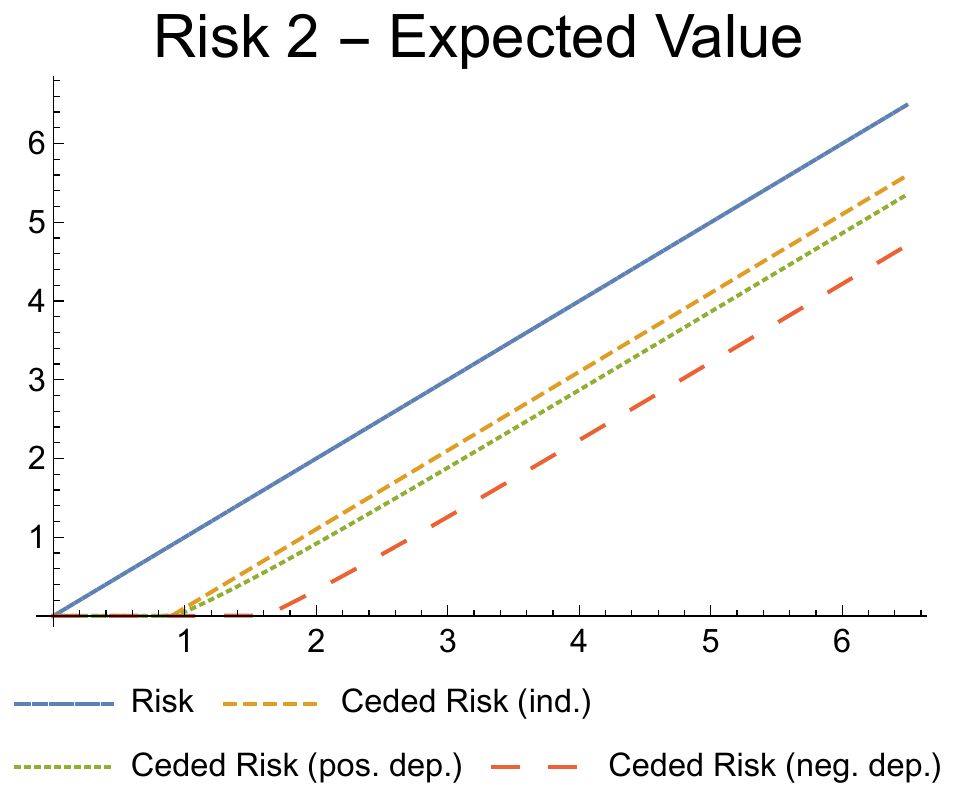}
		\hspace{0.0cm}
		\includegraphics[width=0.24\textwidth]{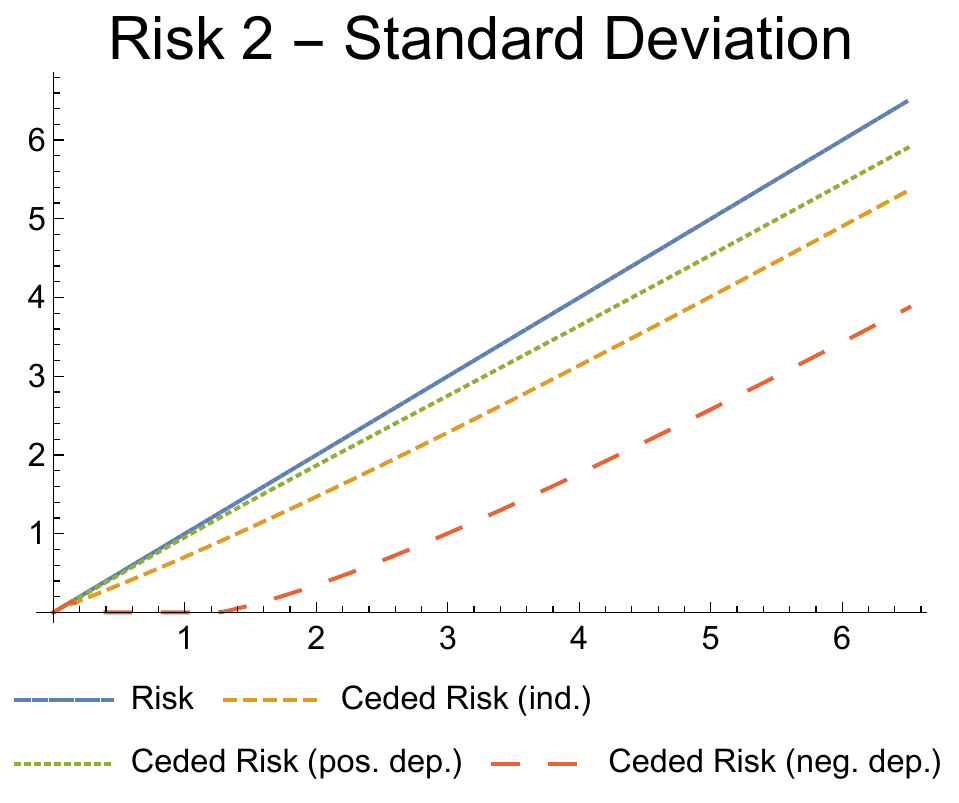}
		\hspace{0.0cm}
		\includegraphics[width=0.22\textwidth]{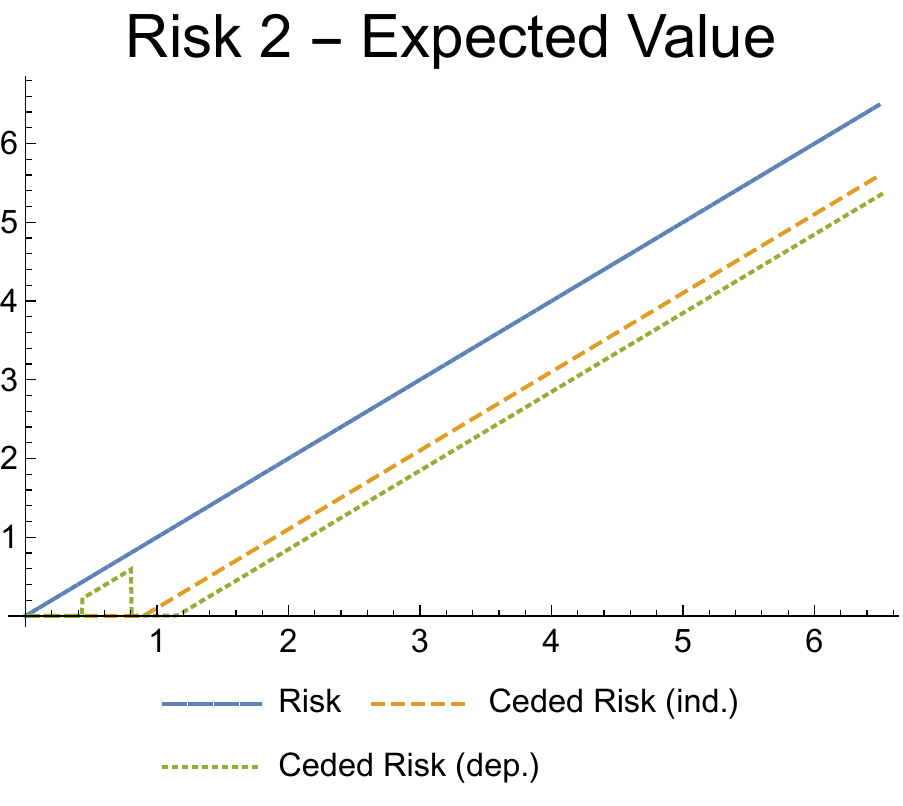}
		\hspace{0.0cm}
		\includegraphics[width=0.22\textwidth]{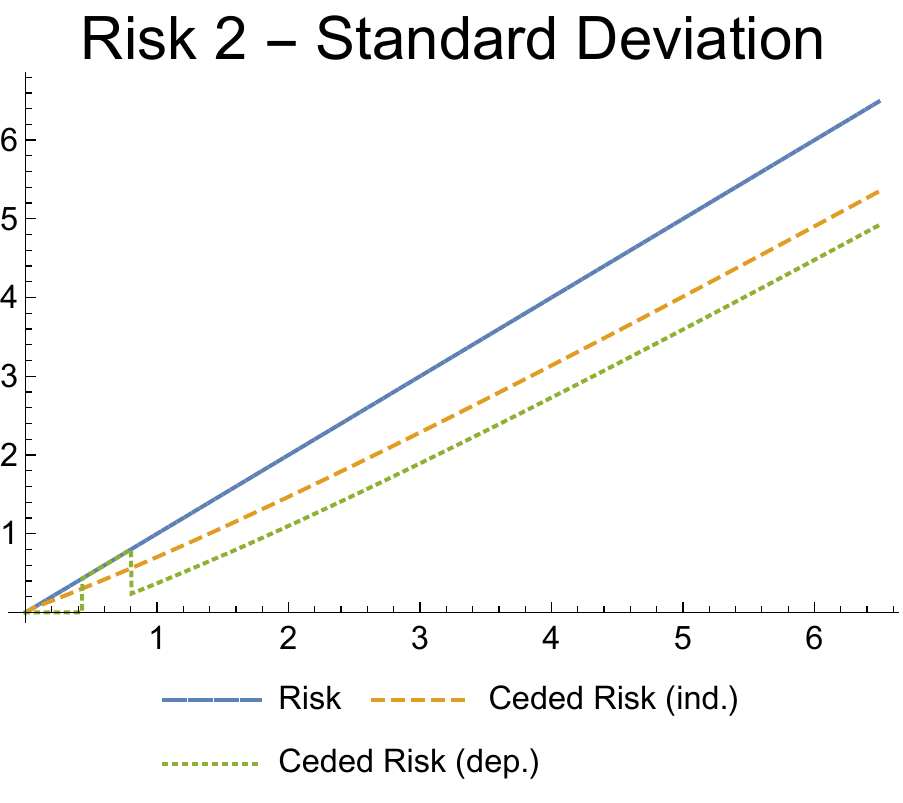}
	\end{center}
\caption{Optimal reinsurance treaties (ceded risk) for risks $X_1$ and $X_2$ for the two premium calculation principles when
(i) the risks are dependent through Frank's copula with $\alpha=10$ (dotted green line) and $\alpha=-10$ (big dashed red line), first and second columns;
(ii) the risks are dependent through a copula including positive and negative dependencies, third and fourth columns.
In blue is 	represented  the underlying risk and in yellow (dashed) the optimal solution for the independent case.
}
\label{results}
\end{figure}

In the case of positive dependence (Figure \ref{results}, first and second columns, dotted green line) we observe that for the expected value principle, the optimal treaty for risk $X_1$ (light tailed) is a decreasing function of the retained risk (notice that the plot represents the ceded risk).
In this case, the optimal treaty under independence is the stop loss. 
Regarding the standard deviation principle, the optimal treaty under dependence has similar behaviour to that of the independent  case, but the ceded risk is higher when positive dependencies are present.

When the risks are negatively dependent (Figure \ref{results}, first and second columns, big dashed red line), the optimal ceded risk is always lower than the optimal ceded risk under independence, whatever the premium principle and for both
risks. 
When the standard deviation principle is considered, the optimal ceded risk is a non-monotonic function.
It is zero on a neighbourhood of the expected value, being the minimum between a convex function and the claim amount. 

In the case of more complex dependencies, including negative and positive dependencies along the domain (Figure \ref{results}, third and fourth columns), the optimal treaty exhibits abrupt swings in the amount of ceded risk. 
For the light-tailed risk ($X_1$), the optimal treaty cedes all the risk at small claim amounts and cedes no risk on the tail, with some risk ceded for losses between the expected value and approximately 1.5 times the expected value.
This pattern persists for the two premia.
For the heavy-tailed risk ($X_2$) it is optimal to cede some risk for losses below the expected value, while in the tail the ceded risk is similar, but below, to the independent case.
Obviously, the behaviour of the ceded risk in this case is highly dependent on the dependence structure and this example serves to show that dependencies may lead to unexpected optimal treaties, specially when more intricate dependencies are at stake.

\appendix

\section{Example of an optimality criterion which is concave in $\mathcal Z$ and nonconcave in $\mathcal H_X$}\label{Appendix concavity}

In this appendix, we show that convexity over $\mathcal H_X$ of the premia does not imply concavity of the functional $\nu \mapsto \mathbb E^\nu U(L_\nu)$, even when $U$ is strictly concave.

For convenience, we consider the case when $X_1$ is exponentially distributed with parameter $\lambda>0$.
Further, assume that the utility function is exponential, i.e.,
\[
U(x)= - e^{-Rx}, \qquad x \in \mathbb R,
\]
and the premium for risk $1$ is computed by the expected value principle, $P_1(\nu) = (1+\theta)\mathbb E^\nu Z_1$ ($R, \theta>0$ are constant parameters).
Consider two deterministic strategies
\begin{align*}
&
Z_1(x) = (x-m)^+, \qquad Z_i(x)= x \quad i=2, \ldots , n;
\\ &
\tilde Z_1(x) = x \chi_{[0,a]}(x) + (x-b)^+, \qquad \tilde Z_i(x)= x \quad i=2, \ldots , n ,
\end{align*}
where $m$, $a$ and $b$ are positive parameters with $a<b$.
These treaties correspond to the measures $\nu, \eta \in \mathcal H_X$ such that
\begin{align*}
\nu (dx,dz) = &
\left(
\delta_{(x_1-m)^+}(dz_1) \times \bigotimes\limits_{i=2}^n \delta_{x_i}(dz_i)
\right) \mu_X(dx),
\\
\eta (dx,dz) = &
\left(
\left(\chi_{[0,a]}(x_1)\delta_{x_1}+ \chi_{]a,b[}(x_1) \delta_0 + \chi_{[b,+\infty[}(x_1)\delta_{(x_1-b)^+}\right)(dz_1) \times \bigotimes\limits_{i=2}^n \delta_{x_i}(dz_i)
\right) \mu_X(dx),
\end{align*}
respectively.
In this setting,
\begin{align*}
\mathbb E^\nu U(L_\nu) = &
- e^{R \left(\sum \limits_{i=2}^n P_i(\nu) - c \right)}
e^{R P_1(\nu)}\mathbb E^{\mu_{X_1}} \left(e^{R(X_1-Z_1(X_1))}\right)
\\
\mathbb E^\eta U(L_\eta) = &
- e^{R \left(\sum \limits_{i=2}^n P_i(\nu) - c \right)}
e^{R P_1(\eta)}\mathbb E^{\mu_{X_1}} \left(e^{R(X_1-\tilde Z_1(X_1))}\right).
\end{align*}
Further, for any $t \in ]0,1[$:
\begin{align*}
&
\mathbb E^{(1-t)\nu+ t \eta} U(L_{(1-t)\nu+ t \eta}) =
\\ = & 
- e^{R \left(\sum \limits_{i=2}^n P_i(\nu) - c \right)}
e^{R \left((1-t) P_1(\nu) + t P_1(\eta)\right)} \times
\\ & \times \left(
(1-t) \mathbb E^{\mu_{X_1}} \left(e^{R(X_1-Z_1(X_1))}\right)
+
t \mathbb E^{\mu_{X_1}} \left(e^{R(X_1-\tilde Z_1(X_1))}\right)
\right) 
\\ = &
- e^{R \left(\sum \limits_{i=1}^n P_i(\nu) - c \right)}
e^{t R\left( P_1(\eta) - P_1(\nu) \right)} \times
\\ & \times \left(
\mathbb E^{\mu_{X_1}} \left(e^{R(X_1-Z_1(X_1))}\right)
+
t \mathbb E^{\mu_{X_1}} \left(e^{R(X_1-\tilde Z_1(X_1))} - e^{R(X_1-Z_1(X_1))} \right)
\right) .
\end{align*}
We will show that the second derivative (with respect to $t$) of the last expression, can be positive or negative depending on the choice of the parameters $m$, $a$, and $b$, and therefore the optimality criterion is neither concave nor convex in $\mathcal H_x$.

Suppose that $R \neq \lambda$, and let
\begin{align*}
A = &
R(P_1(\eta)-P_1(\nu))=
R(1+\theta) \frac{e^{-\lambda a}\left( e^{\lambda a} -1 - \lambda a \right) + e^{-\lambda b} - e^{-\lambda m}} {\lambda};
\\
B = &
\mathbb E^{\mu_{X_1}}\left( e^{R(X_1-Z_1(X_1))} \right) =
\frac R{R-\lambda} e^{(R-\lambda) m} - \frac \lambda{R-\lambda};
\\
C = &
\mathbb E^{\mu_{X_1}}\left( e^{R(X_1-\tilde Z_1(X_1))} \right) =
1 - e^{-\lambda a} - \frac \lambda{R-\lambda} e^{(R-\lambda) a} + \frac R{R-\lambda} e^{(R-\lambda) b} .
\end{align*}
Since $\frac{d^2}{d t^2}\left(e^{At} \left( B +(C-B) t \right) \right) = e^{At} A \left( AB+2(C-B) +A(C-B) t \right)$, we only need to study the sign of the quantity $A(AB+2(C-B))$.

We introduce the new parameter $\varepsilon = \mathbb E^{\mu_{X_1}}\left( \tilde Z_1(X_1) - Z_1(X_1) \right)$, that is
\begin{align*}
e^{-\lambda m} = e^{-\lambda a} \left( e^{\lambda a} - 1 - \lambda a\right) + e^{-\lambda b} - \lambda \varepsilon .
\end{align*}
Therefore, for small $a>0$ and small $\varepsilon \in \mathbb R$, we have
\begin{align*}
A = &
R(1+\theta) \varepsilon;
\\
B = &
\frac{R}{R-\lambda}e^{(R-\lambda)b} -
\frac{\lambda}{R-\lambda} -
\frac{\lambda R}{2}e^{Rb} a^2 +
R e^{Rb} \varepsilon +
o(a^2+|\varepsilon| );
\\
C = &
\frac{R}{R-\lambda}e^{(R-\lambda)b} -
\frac{\lambda}{R-\lambda} -
\frac{\lambda R}{2} a^2 +
o(a^2 ).
\end{align*}
Thus,
\begin{align*}
&
A(AB+2(C-B)) =
\\ = &
R(1+\theta ) \varepsilon \left(
R(1+\theta) \varepsilon \left( \frac{R}{R-\lambda}e^{(R-\lambda)b} -
\frac{\lambda}{R-\lambda} \right) +
2 \left( \frac{\lambda R}2 \left(e^{Rb}-1\right) a^2 -Re^{Rb} \varepsilon \right)
\right) +
\\ & +
o(a^4+ \varepsilon^2) 
\\ = &
R^2e^{Rb}(1+\theta ) \varepsilon \left(
\left(
(1+\theta) \left( \frac{R e^{-\lambda b}}{R-\lambda} -
\frac{\lambda e^{-R b}}{R-\lambda} \right)
-2
\right) \varepsilon +
\lambda \left(1-e^{-Rb} \right) a^2  \right) +
o(a^4+ \varepsilon^2) .
\end{align*}
Fix arbitrary $b>0$ and small $a \in ]0,b[$.
Since $\lambda>0$ and $R>0$, 
$\left(
(1+\theta) \left( 
\frac{R e^{-\lambda b}}{R-\lambda} -
\frac{\lambda e^{-R b}}{R-\lambda} \right)
-2 \right) 
\varepsilon +
\lambda \left(1-e^{-Rb} \right) a^2  >0$ 
for every $\varepsilon$ sufficiently small.
Thus, $A(AB+2(c-B))$ has the same sign as $\varepsilon$.

\vspace{0.5cm}
\noindent \textbf{Acknowledgments}

\noindent The authors acknowledge financial support from FCT -- Funda\c c\~ao para a Ci\^encia e Tecnologia (Portugal), national funding, through research grant UIDB/05069/2020.


\bibliographystyle{apalike}

\end{document}